\documentclass{amsart}

\usepackage{graphicx}%

\usepackage{amssymb,amsmath,latexsym,amsthm,amsfonts,bbm,dsfont}
\usepackage[colorlinks=true,citecolor=red,linkcolor=blue,pdfpagetransition=Blinds]{hyperref}
\usepackage{amsthm}%
\usepackage{mathscinet}
\usepackage{amsmath, amssymb}
\usepackage{algorithm}
\usepackage[noend]{algpseudocode}
\usepackage{hyperref}
\usepackage{anysize}
\usepackage{algorithm}
\marginsize{3cm}{3cm}{3cm}{3cm}
\usepackage{xcolor}%
\usepackage{mathtools,stmaryrd}
\numberwithin{equation}{section}
\DeclareMathOperator*{\esssup}{ess\,sup}
\newtheorem{theorem}{Theorem}[section]
\newtheorem{proposition}[theorem]{Proposition}
\newtheorem{lemma}[theorem]{Lemma}
\newtheorem{corollary}[theorem]{Corollary}
\newtheorem{remark}[theorem]{Remark}

\begin{document}
\title[Lipschitz stability and reconstruction in inverse problem]{Lipschitz stability and reconstruction in inverse problems for semi-discrete parabolic operators}
\author[R. Lecaros]{Rodrigo Lecaros}
\address[R. Lecaros]{Departamento de Matem\'atica, Universidad T\'ecnica Federico Santa Mar\'ia,  Santiago, Chile.}
\email{rodrigo.lecaros@usm.cl}

\author[J. L\'opez-R\'ios]{Juan L\'opez-R\'ios}
\address[J. L\'opez-R\'ios]{Universidad Industrial de Santander, Escuela de Matemáticas, A.A. 678, Bucaramanga, Colombia}
\email{jclopezr@uis.edu.co}

\author[A. A. P\'erez]{Ariel A. P\'erez}
\address[A. A. P\'erez]{(Corresponding Author) Departamento de Matem\'atica, Universidad del B\'io-B\'io, Concepci\'on, Chile.}
\email{aaperez@ubiobio.cl}


\begin{abstract}
    This work addresses an inverse problem for a semi-discrete parabolic equation, consisting of identifying the right-hand side of the equation from solution measurements at an intermediate time and within a spatial subdomain. We apply this result to establish a stability estimate for a coefficient inverse problem involving the recovery of a spatially dependent potential function. Furthermore, we present a reconstruction algorithm for recovering this coefficient and provide a proof of its convergence. Our approach relies on a novel semi-discrete Carleman estimate in which the parameter is constrained by the mesh size. Due to the discrete terms arising in the Carleman inequality, this method naturally introduces an error term associated with the solution's initial condition. 

\end{abstract}
\maketitle

\section{Introduction}
  Let $d\geq 1$, $T>0$, and $\displaystyle\Omega:= \prod_{i=1}^{d}(0,1)\subset \mathbb{R}^d$, with $\omega\Subset \Omega$ denoting an arbitrary subdomain. We consider the following parabolic system
\begin{equation}\label{system:continuous}
\begin{cases}
\partial_{t} y(t,x) - \mathcal{A}y(t,x)= g(t,x), & (t,x) \in (0,T)\times\Omega,\\
y(t,x) = 0, & (t,x) \in (0,T)\times\partial\Omega,\\
y(0,x) = y_{ini}(x), & x \in \Omega,
\end{cases}
\end{equation}
where $\mathcal{A}$ is a uniformly elliptic second-order  operator defined by
\begin{equation}\label{ope:A:continuous}
\mathcal{A}y(t,x)=\sum_{i=1}^{d}\frac{\partial}{\partial x_i}\left(\gamma_{i}(t,x)\frac{\partial y}{\partial x_{i}}(t,x)\right)-\sum_{i=1}^{d}b_{i}(t,x)\frac{\partial y}{\partial x_{i}}(t,x)-c(t,x)y(t,x).
\end{equation} 
Here, $\gamma_{i}(t,x)>0$ for all $(t,x)\in(0,T)\times\Omega$, and $g\in H^1((0,T),L^2(\Omega))$.

In this framework, a classical inverse problem consists of determining the source term $g(t,x)$ from observations of $y$ within the subdomain $\omega$. Specifically, for a fixed time $\vartheta\in(0,T)$, we consider the observation operator $\Lambda_{\vartheta}:H^1((0,T),L^2(\Omega))\to H^2(\Omega)\times H^1((0,T),L^2(\omega)),$ given by
$$\Lambda_{\vartheta}(g):= (y|_{t=\vartheta},y|_{\omega\times(0,T)}),$$
where $y$ denotes the solution of \eqref{system:continuous}. The stability of the inverse problem corresponds to the Lipschitz inequality
\begin{equation}\label{stability:continuous}
    \|g\|_{H^1((0,T),L^2(\Omega))}\leq C\|\Lambda_{\vartheta}(g)\|:= C\left(\|y|_{t=\vartheta}\|_{H^2(\Omega)}+\|y\|_{H^1((0,T),L^2(\omega))}\right),
\end{equation}
for some constant $C>0$.

Several works have addressed this inverse problem in the literature; see, for instance, \cite{IY-2024,IY-2024-II,IY-1998,Yamamoto_2001}. As noted in \cite{IY-2024}, most results in this area are obtained when the observation time $\vartheta$ is lies in $(0,T)$, following the method introduced by Bukhgeim and Klibanov \cite{bukhgeim1981global,MetodoB-K1981,Klibanov1984}. In \cite{IY-1998}, the authors applied this method to prove uniqueness and Lipschitz stability of the inverse problem, while in \cite{IY-2024-II}, they established conditional Lipschitz stability and uniqueness for the case $\vartheta=T$.

\par In contrast, the (semi)discrete setting has been explored primarily in the context of controllability problems for parabolic operators; see \cite{BLR-2014,CLNP2022,N:2014} for the spatial semi-discrete setting, \cite{BHS2020} for the time semi-discrete case, and \cite{GCHS2021,LPP:2024} for the fully discrete setting. Recently, the time semi-discrete setting for an inverse problem was studied in \cite{Klibanov2024_tiempo_discreto}. In this regard, the authors did not discuss extending the analysis to the spatial semi-discrete framework for parabolic operators. Moreover, the only results on inverse problems in the spatial semi-discrete setting are \cite{BEO-2015,ZZ-2022} and \cite{ZZ-2023}, concerning the wave and Schrödinger equations, respectively. Hence, our objective is to fill this gap by studying an inverse source problem for a spatial semi-discretization of the system \eqref{system:continuous}, establishing its stability (see Theorem \ref{theo:stability}), and providing a reconstruction algorithm to recover a zeroth-order spatially dependent coefficient (see Algorithm \ref{Alporithmo_Reconstruccion}).

Let us introduce some notation related to the spatial semi-discrete framework.  Given $N\in\mathbb{N}$, let $h=\frac{1}{N+1}$ be small enough to represent the size of the mesh. We define the Cartesian grid of $[0,1]^{d}$ as
\begin{equation}\label{mesh}
    \mathcal{K}_{h}:= \left\{x\in [0,1]^{d}\mid \exists k\in\mathbb{Z}^{d} \text{ such that }x=hk  \right\}.
\end{equation} 
We set the mesh $\mathcal{W}:= \Omega\cap \mathcal{K}_{h}$ and denote by $C(\mathcal{W})$ the set of functions defined on $\mathcal{W}$. 
Moreover, we define the average and difference operators as 
\begin{equation}\label{op:difference:average}
\begin{split}
    A_{i}u(x)&:= \frac{1}{2}\left(\tau_{+}i u(x)+\tau_{-i}u(x)\right),\\
    D_{i}u(x)&:= \frac{1}{h}\left(\tau_{+i}u(x)-\tau_{-i}u(x) \right),
    \end{split}
\end{equation}
where $\tau_{\pm i}y(x):= y(x\pm \frac h2 e_{i})$, being $\{e_{i}\}_{i=1}^{d}$ the canonical basis of $\mathbb{R}^{d}$. Thus, by denoting $Q:= (0,T)\times\mathcal{W}$, the spatial semi-discrete approximation of the system \eqref{system:continuous} is given by
\begin{equation}\label{system:discrete}
\begin{cases}
    \partial_{t}y(t,x)-\mathcal{A}_{h}y(t,x)=g(t,x),\quad &(t,x)\in Q,\\
    y(t,x)=0,  &(t,x)\in(0,T)\times\partial\mathcal{W},\\
    y(0,x)=y_{ini}(x), &x\in\mathcal{W},
\end{cases}
\end{equation}
with $\mathcal{A}_{h}$ being the finite difference space approximation of the continuous operator \eqref{ope:A:continuous}, given by
\begin{equation}
\mathcal{A}_{h}y(t,x):= \sum_{i=1}^{d}D_{i}\left(\gamma_{i}(t,x)D_{i}y(t,x)\right)-\sum_{i=1}^{d}b_i(t,x)D_{i}A_{i}y(t,x)-c(t,x)y(t,x).
\end{equation}
Our inverse problem consists of determining the right-hand side of the system \eqref{system:discrete}, known as an inverse source problem, from the knowledge of the data $\left(y\Bigr|_{t=\vartheta},\Bigl.y\Bigr|_{(0,T)\times \omega}\right)$, where $\omega\subset\mathcal{W}$ is an arbitrary subdomain. That is, we investigate the semi-discrete analogue of \eqref{stability:continuous}. 

 Assume that $\gamma_i,b_i\in C^1([0,T]\times \overline{\Omega})$, for all $i=1,\ldots,d$, that $c\in C^1([0,T]\times \overline{\Omega})$ and that $y_{ini}\in C^0(\overline{\Omega})$. Let
$$\Gamma(t,x):= \mathrm{Diag}(\gamma_{1}(t,x),\gamma_{2}(t,x),\ldots, \gamma_{d}(t,x)),$$ such that $\gamma_{i}(t,x)>0$ for all $i=1,\ldots,d$, and it holds
\begin{equation*}
    \mbox {reg}(\Gamma):= \esssup_{
    \begin{array}{c}
    (t,x)\in [0,T]\times\overline{\Omega}\\
    i=1,\ldots,d    
    \end{array}
    }\left( \gamma_{i}(t,x)+\frac{1}{\gamma_{i}(t,x)}+|\nabla_{x}\gamma_{i}(t,x)|+|\partial_t\gamma_{i}(t,x)|\right)<+\infty.
\end{equation*} 
Given $\mbox{reg}^{0}>0$, henceforth $\Gamma$ is such that $\mathrm{reg}(\Gamma)\leq\mbox{reg}^{0}$. 

Moreover, assume that, for some constant $C>0$, the function $g\in C^1([0,T];L^\infty(\overline{\Omega}))$ satisfies the estimate
    \begin{equation}\label{assump:g}
        |\partial_{t}g(t,x)|\le C|g(\vartheta,x)|, \quad \text{for almost all }(t,x)\in(0,T)\times\overline{\Omega}.
    \end{equation} 
Our first main result is the following stability estimate. The detailed notation is introduced in the next section. 
\begin{theorem}\label{theo:stability}
    	 Let $\psi$ satisfy \eqref{assumtion:psi} and let $\varphi$ be given by \eqref{funcion-peso-2}. Assume that $g$ satisfies \eqref{assump:g}, and let $y\in \mathcal{C}^{1}([0,T],\overline{\mathcal{W}})$ be the solution of the system \eqref{system:discrete}. Then there exist positive constants $C$, $C''$, $\tau_{0}\geq 1$, $h_{0}>0$, $\varepsilon >0$, depending on $\omega$, $\mathrm{reg}^{0}$ and $T$, such that for all $\tau\geq \tau_{0}(T+T^{2})$, $0<h\leq h_{0}$, there exists $0<\delta(h) \leq 1/2$ satisfying $\tau h(\delta T^{2})^{-1}\leq \varepsilon$, and the estimate
		\begin{equation*}
        \begin{aligned}
		    \|g\|_{L_h^{2}(Q)}\leq &Ce^{\frac{C''}{T^2}\tau}\left(\|y|_{t=\vartheta}\|_{H_h^2(\mathcal{W})}+\|e^{s\varphi}\partial_{t}y\|_{L_h^2(Q_\omega)}+\|e^{s\varphi}y\|_{L_h^2(Q_\omega)}\right)\\
            &+Ce^{-\frac{C''}{h}}\left(\|y|_{t=0}\|_{L^2_h(\mathcal{W})}+\|\partial_{t}y|_{t=0}\|_{L^2_h(\mathcal{W})}\right),
            \end{aligned}
		\end{equation*} 
		  holds for $Q_{\omega}:= (0,T)\times\omega$.
\end{theorem}

In the inequality of the above Theorem, there is an error term $$e^{-\frac{C''}{h}}\left(\|y|_{t=0}\|_{L^2_h(\mathcal{W})}+\|\partial_{t}y|_{t=0}\|_{L^2_h(\mathcal{W})}\right),$$ which arises from the discrete phenomenon and tends to zero as $h\to 0$. Moreover, if we assume $y(0)=\partial_ty(0)=0$, we recover the classical inequality for the continuous case as in \cite{IY-1998}.

The proof of Theorem \ref{theo:stability} is based on the new Carleman estimate \eqref{ine:carleman:semi-discreteNew} established for the operator in \eqref{system:discrete}. To our knowledge, the only known Carleman estimate available in the literature for semi-discrete parabolic operators in arbitrary dimensions is that of \cite{BLR-2014}. However, it is not suitable for studying the inverse problem because it lacks a  term involving the second-order spatial operator. In this work, we address this issue by establishing Carleman estimates for the solution of system \eqref{system:discrete} and \eqref{system:z}, corresponding to the cases $q=0$ and $2q=1$, respectively. These results are summarized below.

\begin{theorem}\label{theo:Carleman}
Let $\psi$ satisfy \eqref{assumtion:psi} and let $\varphi$ be given by \eqref{funcion-peso-2}. Let $y\in \mathcal{C}^{1}([0,T],\overline{\mathcal{W}})$ be the solution of system \eqref{system:discrete}. For $\lambda\geq 1$ sufficiently large, there exist constants $C$, $\tau_{0}\ge 1$, $h_{0}>0$, $\varepsilon >0$, depending on $\omega$, $\omega_{0}$, $\mathrm{reg}^{0}$, $T$, and $\lambda$, such that
    \begin{align}\label{ine:carleman:semi-discreteNew}
		I_{2q}(y)+J_{2q}(y) 
		\leq&  C\left(\int_Qe^{2\tau\theta\varphi}(\tau\theta)^{2q}|g|^2+\int_{(0,T)\times\omega}(\tau\theta)^{{2q}+3}e^{2\tau\theta\varphi}|y|^2\right)\\
		&+Ch^{-2}\int_{\mathcal{W}} (\tau\theta(0))^{2q}\left(\Big|y|_{t=0}\Big|^2+\Big|y|_{t=T}\Big|^2\right)e^{2\tau\theta(0)\varphi},\notag
	\end{align}
where 
\begin{align*}
     I_{{2q}}(y):= \int_Q(\tau\theta)^{2q-1}&|\partial_{t}y|^2e^{2\tau\theta\varphi}+\sum_{i,j\in\llbracket 1,d\rrbracket}\int_{Q^\ast_{ij}}(\tau\theta)^{2q-1}\gamma_{i}\gamma_{j}e^{2\tau\theta\varphi}|D_{ij}y|^{2},
 \end{align*}
 and
\begin{align*}
J_{2q}(y):= &\tau^{2q+1}\sum_{i\in\llbracket 1,d\rrbracket}\left( \left\|\theta^{1/2+q} e^{\tau\theta\varphi} D_{i}y\right\|^{2}_{L^{2}_{h}(Q^{\ast}_i)}+\left\|\theta^{1/2+q} e^{\tau\theta\varphi}A_{i}D_{i}y\right\|^{2}_{L^{2}_{h}(Q)}\right)\\
&+\tau^{3+2q}\left\|\theta^{3/2+q}e^{\tau\theta\varphi}y\right\|^{2}_{L^{2}_{h}(Q)},
\end{align*} 
for any $q\in\mathbb{R}$, $\tau\geq \tau_{0}(T+T^{2})$, $0<h\leq h_{0}$, $0<\delta \leq 1/2$, $\tau h(\delta T^{2})^{-1}\leq \varepsilon$.
\end{theorem}
Finally, assume that the coefficients $\gamma_i, b_i$, $i=1,\ldots,d$, and $c$, are independent of time. Our final main result is a reconstruction algorithm for the inverse problem of identifying the potential $p\in L^\infty_h({\mathcal{W}})$ from the measurements of $y$, the solution of
\begin{equation}\label{system:coefficient:02}
\begin{cases}
    \partial_{t}y(t,x)-\mathcal{A}_{h}y(t,x)+p(x)y(t,x)=g(t,x),\quad &(t,x)\in Q,\\
    y(0,x)=y_{ini}(x), &x\in\mathcal{W},\\
    y(t,x)=f(t,x), &(t,x)\in(0,T)\times\partial\mathcal{W},
\end{cases}
\end{equation}
where $f,g$ and $y_{ini}$ are given functions. 

More precisely, we consider the measurement operator 
 $$\Lambda_p:= (y_p|_{t=T/2},\partial_ty_p|_{Q_{\omega}},\partial_ty_p|_{\{t=0\}\times\mathcal{W}}),$$
 where $y_p$ is the solution of \eqref{system:coefficient:02} associated with $p$. 
 
We also define, for $m>0$, the set $$\mathcal{X}_{m}:= \{p\in L^{\infty}_h(\mathcal{W}):\;\|p\|_{L^{\infty}_h}\leq m\}.$$

Given $p^{\ast}\in \mathcal{X}_{m}$, assume that there exists $\alpha>0$ such that $\Big|y_{p^\ast}|_{t=T/2}\Big|>\alpha$. Considering the functional $\mathcal{J}_{\tau,p}$ defined in \eqref{FunctionalJ}, we can reconstruct the coefficient $p^{\ast}$ from the measurements $\Lambda_{p^{\ast}}$ through the following iterative scheme. 

\begin{algorithm}\label{Alporithmo_Reconstruccion}
\caption{Iterative Reconstruction of $p$}
\label{alg:pk_reconstruction}
\begin{algorithmic}
  \State \textbf{Initialization:}
  \State Set $p_{0}=0$
  \Statex

  \For{$k=0,1,2,\ldots$ until convergence}
    \State \textbf{Step 1: Forward solve}
    \State Compute $y_{p_k}$, the solution of \eqref{system:coefficient:02} with $p=p_k$.
    
    \State \textbf{Step 2: Residuals}
    \State Set 
    \[
      \mu_k :=  \partial_t y_{p_k} - \partial_t y_{p^\ast}
      \quad \text{on } Q_\omega,
    \]
    and
    \[
      \eta_k :=  \partial_t y_{p_k}\Big|_{t=0} - \partial_t y_{p^\ast}\Big|_{t=0}
      \quad \text{on } \mathcal{W}.
    \]
    
    \State \textbf{Step 3: Minimization}
    \State Compute 
    \[
      u^\ast_{p_k} = \operatorname*{argmin}\,\mathcal{J}_{\tau,p_k}[\mu_k,\eta_k]
    \]
    
    \State \textbf{Step 4: Update}
    \State Set
    \[
      \tilde{p}_{k+1} :=  \frac{y_{p_k}\Big|_{t=T/2}}{y_0}\,p_k
      + \frac{u^\ast_{p_k}\Big|_{t=T/2} - \mathcal{A}_h y_{p_k}\Big|_{t=T/2} + \mathcal{A}_h y_0}{y_0},
    \]
    where $y_0 = y_{p^\ast}\Big|_{t=T/2}$
    
    \State \textbf{Step 5: Projection}
    \State Set
    \[
      p_{k+1} = T_m(\tilde{p}_{k+1}),
    \]
    where
    \[
      T_m(x) := 
      \begin{cases}
        x, & |x|\leq m,\\
        \operatorname{sign}(x)\,m, & |x|> m.
      \end{cases}
    \]
  \EndFor
\end{algorithmic}
\end{algorithm}

The convergence of the above algorithm for $\tau$ large enough is guaranteed by the following result.

\begin{theorem}\label{theo:convergence} Let $m>0$, $p^\ast\in\mathcal{X}_{m}$ and assume that $\gamma_i, b_i$, $i=1,\ldots,d$, and $c$ to be independent of time. Assume that  \eqref{assumtion:psi} and \eqref{funcion-peso-2} hold, that there exists $\alpha>0$ such that
$\Big|y_{p^{\ast}}|_{t=T/2}\Big|>\alpha$, where $y_{p^\ast}$ is the solution of \eqref{system:coefficient:02} with $p=p^\ast$. Then, there exists $M > 0$ such that for any $\tau\geq \tau_{0}(T+T^{2})$, $0<h\leq h_{0}$, there exists $0<\delta(h) \leq 1/2$ satisfying $\tau h(\delta T^{2})^{-1}\leq \varepsilon$, and the estimate
$$\int_{\mathcal{W}}e^{2\tau\theta(T/2)\varphi}|p_{k+1}-p^\ast|^2\leq M \tau^{-\frac{3}{2}}\int_{\mathcal{W}}e^{2\tau\theta(T/2)\varphi}|p_k-p^\ast|^2,$$ holds for all $k\in\mathbb{N}$.
In particular, for $\tau$ is large enough, the above algorithm converges.
\end{theorem}
The remainder of the paper is organized as follows. Section \ref{sec:proof:carleman} introduces the notation and preliminaries to be used throughout the paper, followed by the proof of the Carleman estimate stated in Theorem \ref{theo:Carleman}. Section \ref{sec:inverse:problem} is intended to study the stability estimate and the analysis of the inverse problem. Section \ref{sec:stability:reconstruction} is devoted to the proof of the stability and convergence of the reconstruction algorithm. Finally, concluding remarks and future perspectives are discussed in Section \ref{sec:concluding}.

\section{A new Carleman estimate for a semi-discrete parabolic operator}\label{sec:proof:carleman}

\subsection{Some preliminary notation}
In this section, we complement the notation for meshes and operators given in the previous section. Recall that $\mathcal{W}:= \Omega\cap\mathcal{K}_{h}$, where $\mathcal{K}_{h}$ is defined in \eqref{mesh}. Then, using the translation operators $\tau_{\pm i}(\mathcal{W}):= \left\{ x\pm\frac{h}{2}e_{i}\mid x\in \mathcal{W}\right\}$, we define the dual mesh in direction $i$  
\begin{equation}\label{sets:dual}
    \mathcal{W}_{i}^{\ast}:= \tau_{+i}\left( \mathcal{W}\right)\cup\tau_{-i}\left( \mathcal{W}\right).
\end{equation}
For the difference and average operators defined in \eqref{op:difference:average} we have a Leibniz rule for functions defined in $\overline{\mathcal{W}}_{ij}:= (\mathcal{W}_{i}^{\ast})_{j}^{\ast}=\mathcal{W}_{ji}^{\ast\ast}$.
\begin{proposition}[{\cite[Lemma 2.1]{E:DG:2011}}]\label{pro:product}
Given $u,v\in C(\overline{\mathcal{W}})$, the following identities hold in $\mathcal{W}^{\ast}_{i}$. For the difference operator
\begin{equation}\label{eq:difference:product}
    D_{i}(u\,v)=D_{i}u\, A_{i}v+A_{i}u\,D_{i}v,
    \end{equation}
and for the average operator \begin{equation}\label{eq:average:product}
     A_{i}(u\,v)= A_{i}u\,A_{i}v+\frac{h^{2}}{4}D_{i}u\,D_{i}v.
\end{equation}
\end{proposition}
\begin{remark} Several useful consequences follow from \eqref{eq:average:product}; for instance, for the average operator we have
\begin{equation}\label{eq:promediocuadrado}
    A_{i}(|u|^{2})=\left|A_{i}u \right|^{2}+\frac{h^{2}}{4}\left| D_{i}u\right|^{2},
\end{equation}
and 
\begin{equation}\label{eq:promedioinequality}
A_{i}(|u|^{2})\geq \left|A_{i}u\right|^{2}.
\end{equation}
For the difference operator, we have
\begin{equation}\label{eq:derivadacuadrado}
    D_{i}(|u|^{2})=2D_{i}u\,A_{i}u.
\end{equation}
\end{remark}
We now introduce the discrete integration by parts for the operators \eqref{op:difference:average}. Define the boundary of $\mathcal{W}$ in the direction $e_{i }$ as $\partial_{i}\mathcal{W}:= \overline{\mathcal{W}}_{ii}\setminus\mathcal{W}$. Moreover, the boundary of $\mathcal{W}$ is defined by
\begin{equation}
    \begin{split}
      \partial \mathcal{W}&:= \bigcup_{i=1}^{d} \overline{\mathcal{W}}_{ii}\setminus\mathcal{W}.
    \end{split}
\end{equation}
For a given set $\mathcal{W}\subseteq\mathcal{K}_{h}$ and $u\in C(\mathcal{W})$, define the discrete integral as
\begin{equation}
    \int_{\mathcal{W}} u:= h^{d}\sum_{x\in \mathcal{W}} u(x),
\end{equation}
and the corresponding $L^{2}_{h}$ inner product on $C(\mathcal{W})$:
\begin{equation}
    \langle u,v\rangle_{\mathcal{W}}:= \int_{\mathcal{W}}u\,v,\quad \forall u,v\in C(\mathcal{W}),
\end{equation}
with the associated norm
\begin{equation}
    \left\|u\right\|_{L^{2}_{h}(\mathcal{W})}:= \sqrt{\langle u,u\rangle_{\mathcal{W}}}.
\end{equation}
For $u\in C(\mathcal{W})$, define its $L^{\infty}_{h}(\mathcal{W})$ norm by
\begin{equation}
    \left\| u\right\|_{L_{h}^{\infty}(\mathcal{W})}:= \max_{x\in \mathcal{W}}\left\{ |u(x)|\right\},
\end{equation}
and, for $u\in C(\overline{\mathcal{W}})$,
\begin{align}
    \|u\|^2_{H^{1}_{h}(\mathcal{W})}:=\|u\|_{L_{h}^{2}(\mathcal{W})}^2+\sum_{i\in\llbracket 1,d\rrbracket}\int_{\mathcal{W}^\ast_i}|D_{i}u|^{2},\\
    \|u\|^2_{H^{2}_{h}(\mathcal{W})}:=\|u\|_{L_{h}^{2}(\mathcal{W})}^2+\sum_{i\in\llbracket 1,d\rrbracket}\int_{\mathcal{W}}|D^{2}_{i}u|^{2}+|A_{i}D_{i}u|^{2}.
\end{align}
For boundary integrals, given $u\in C(\partial_{i}\mathcal{W})
$, define
\begin{equation}
    \int_{\partial_{i}\mathcal{W}} u:= h^{d-1}\sum_{x\in\partial_{i}\mathcal{W}} u(x).
\end{equation}

Finally, for boundary points, define the exterior normal to $\mathcal{W}$ in the direction $e_{i}$ as $\nu_{i}\in C(\partial\mathcal{W}_{i})$:
\begin{equation}
    \forall x\in\partial_{i}\mathcal{W}, \nu_{i}(x):= \begin{cases} \ \ 1 &  \mbox{if }\tau_{-i}(x)\in \mathcal{W}_{i}^{\ast} \mbox{ and }\tau_{+i}(x)\notin\mathcal{W}_{i}^{\ast},\\ 
    -1 &  \mbox{if }\tau_{-i}(x)\notin \mathcal{W}_{i}^{\ast} \mbox{ and }\tau_{+i}(x)\in\mathcal{W}_{i}^{\ast},\\ 
    \ \ \,0 &  \mbox{elsewhere}.
    \end{cases}
\end{equation}
 We also define the trace operator $t_{r}^{i}$ for $u\in C(\mathcal{W}^{\ast}_{i})$ as
\begin{equation}
    \forall x\in \partial_{i}\mathcal{W},\ t_{r}^{i}(u)(x):= \begin{cases}
    u(\tau_{-i}(x)),\ & \nu_{i}(x)=1,\\
    u(\tau_{+i}(x)),& \nu_{i}(x)=-1,\\
     \ \ 0, & \nu_{i}(x)=0.\end{cases}
\end{equation}
Then, by using the previous notation, we have the following discrete integration by parts identities.
\begin{proposition}[{\cite[Lemma 2.2]{LOPD:2023}}]\label{prop:integralbyparts}
For any $v\in C(\mathcal{W}_{i}^{\ast})$, $u\in C(\overline{\mathcal{W}}_{i})$ we have, for the difference operator
\begin{equation}\label{eq:int:dif}
    \int_{\mathcal{W}}u\,D_{i}v=-\int_{\mathcal{W}_{i}^{\ast}}v\,D_{i}u+\int_{\partial_{i}\mathcal{W}}u\,t_{r}^{i}(v)\nu_{i},\\
\end{equation}    
and for the average operator
\begin{equation}\label{eq:int:ave}
    \int_{\mathcal{W}}u\,A_{i}v=\int_{\mathcal{W}_{i}^{\ast}}v\,A_{i}u-\frac{h}{2}\int_{\partial_{i} \mathcal{W}}u\,t_{r}^{i}(v).
\end{equation}
\end{proposition}

\subsection{On the Carleman weight function}
We introduce the classical weight function used for the semi-discrete parabolic operator, that is, we consider the weight function used in \cite{BLR-2014} and also used in \cite{BHS2020,CLNP2022,GCHS2021,HS2023,LMZP2023}.\\
\textbf{Assumption:} Let $\overline{\omega_0}\subset \omega$ be an arbitrary fixed subdomain of $\Omega$. Let $\widehat{\Omega}$ be a smooth, open, and connected neighborhood of $\overline{\Omega}$ in $\mathbb{R}^d$. The function $x \mapsto \psi(x)$ belongs to $\mathcal{C}^p(\widehat{\Omega}, \mathbb{R})$, for sufficiently large $p$, and satisfies, for some $c > 0$,
\begin{equation}\label{assumtion:psi}
\psi > 0 \quad \text{in } \widehat{\Omega}, \quad |\nabla \psi| \geq c \quad \text{in } \widehat{\Omega} \setminus \omega_0, \quad \text{and} \quad \partial_{n_i} \psi(x) \leq -c < 0 \quad \text{for } x \in V_{\partial_i \Omega},
\end{equation}
where $V_{\partial_i \Omega}$ is a sufficiently small neighborhood of $\partial_i \Omega$ in $\widehat{\Omega}$, where the outward unit normal $n_i$ to $\Omega$  extends from $\partial_i \Omega$.

For $\lambda\geq 1$ and $K>\|\psi\|_{\infty}$, we introduce the functions
\begin{align}\label{funcion-peso-2}
\varphi(x)=e^{\lambda\psi(x)}-e^{\lambda K}<0,
\end{align}
 and, for $0<\delta \leq 1/2$,
\begin{equation}\label{theta-delta}
    \theta(t)=\frac{1}{(t+\delta T)(T+\delta T-t)},\quad t\in [0,T].
\end{equation}

Given $\tau\geq 1$, we set
\begin{equation}\label{eq1}
s(t)=\tau\theta(t).
\end{equation}

\begin{remark}
The parameter $\delta$, chosen so that $0<\delta\leq\frac{1}{2}$, avoids singularities at time $t=0$ and $t=T$. Notice that
\begin{equation}\label{eq:theta} 
\underset{t\in [0,T]}{\max} \theta(t)=\theta(0)=\theta(T)=\frac{1}{T^{2}\delta(1+\delta)}\leq \frac{1}{T^{2}\delta}, 
\end{equation}
and $\underset{t\in[0,T]}{\min}\theta(t)= \theta(T/2)=\frac{4}{T^{2}(1+2\delta)^2}$. Moreover,
\begin{equation}\label{eq:theta'}\frac{d\theta}{dt}=2\left(t-\frac{T}{2}\right)\theta^2(t).
\end{equation}
\end{remark}

In the case where $\gamma_i$ depends only on $x$, the following semi-discrete Carleman estimate was proved in \cite{BLR-2014}.
\begin{theorem}[c.f. {\cite[Theorem 1.4]{BLR-2014}}] Suppose that $\psi$ satisfies assumption \eqref{assumtion:psi}, and that $\varphi$ is defined according to \eqref{funcion-peso-2}. For $\lambda\geq 1$ sufficiently large, there exist $C$, $\tau_{0}\geq 1$, $h_{0}>0$, $\varepsilon >0$, depending on $\omega$, $\omega_{0}$, $\mbox{reg}^{0}$, $T$, and $\lambda$, such that it holds
\begin{align}\label{ine:carleman:semi-discrete}
\tau^{-1}\left\|\theta^{-1/2}e^{\tau\theta\varphi}\partial_{t}y \right\|^{2}_{L^{2}_{h}(Q)}+J_0(y) \leq & C\left(\left\|e^{\tau\theta\varphi}g\right\|^{2}_{L^{2}_{h}(Q)}+\int_{(0,T)\times\omega}\tau^3\theta^3e^{2\tau\theta\varphi}|y|^2\right)\\
&+Ch^{-2}\int_{\mathcal{W}} \left(\Big|y|_{t=0}\Big|^2+\Big|y|_{t=T}\Big|^2\right)e^{2\tau\theta(0)\varphi},\notag
\end{align}

for all $\tau\geq \tau_{0}(T+T^{2})$, $0<h\leq h_{0}$, $0<\delta \leq 1/2$, $\tau h(\delta T^{2})^{-1}\leq \varepsilon$, and $y\in \mathcal{C}^{1}([0,T];\overline{\mathcal{W}})$ being  solution of \eqref{system:discrete}.
\end{theorem}

We note that, in each set $\mathcal{W}^\ast_i$, $\gamma_i$ is the sampling of the given continuous diffusion coefficient $\gamma_i$ on the dual mesh $\mathcal{W}^\ast_i$. However, other consistent discretization of $\gamma_i$ on the dual meshes are possible, such as averaging the values of $\gamma_i$ sampled on the primal mesh  $\mathcal{W}$.

Let us highlight two main differences between the continuous Carleman estimate for a parabolic operator and its semi-discrete version as in \eqref{ine:carleman:semi-discrete}. The first difference is the additional term on the right-hand side, which is exclusively a discrete phenomenon also observed in other semi-discrete operators; see, for instance, \cite{BEO-2015,ZZ-2022,ZZ-2023}. The second difference is the absence on the left-hand side of a term involving the second-order spatial operator $D^{2}_{ij}$, which is crucial in inverse problems. The only semi-discrete Carleman estimate including this second-order term appears in \cite{BHSDT:2019}, in the one-dimensional case.Concerning this last issue, in higher dimensions, it is possible to incorporate it with a higher power of the Carleman parameter and also to consider the time dependency in diffusive functions $\gamma_i$ as stated in Theorem \ref{theo:Carleman}. 

\begin{proof}[Proof of Theorem \ref{theo:Carleman}] Let us first focus on the case $q=0$. Note that the steps developed in Lemmas 3.4, 3.7, and 3.9 from \cite{BLR-2014} still hold provided that $\partial_{t}\gamma_{i}$ is bounded for $i\in\{1,2,\ldots, d\}$. Hence, the Carleman estimate \eqref{ine:carleman:semi-discrete} holds for $\gamma_{i}\in C^1([0,T]\times\overline{\Omega})$.

Let us now focus on the incorporation of the second-order spatial term $D_{ij}^{2}$. First, from $(\ref{system:discrete})$, one has 
\begin{align}\label{ine:A}
    \tau^{-1}\left\| \theta^{-1/2}e^{\tau\theta\varphi}\mathcal{A}_{h}y\right\|^{2}_{L_{h}^{2}(Q)}\leq2\tau^{-1}\left\|\theta^{-1/2}e^{\tau\theta\varphi}\partial_{t}y \right\|^{2}_{L_{h}^{2}(Q)}+2\tau^{-1}\left\|\theta^{-1/2}e^{\tau\theta\varphi}g\right\|^{2}_{L_{h}^{2}(Q)}.
\end{align}

By denoting
\begin{equation*}
U(y):= \tau^{-1}\left\|\theta^{-1/2}e^{\tau\theta\varphi}\partial_{t}y \right\|^{2}_{L^{2}_{h}(Q)}+\tau^{-1}\left\| \theta^{-1/2}e^{\tau\theta\varphi}\mathcal{A}_{h}y\right\|^{2}_{L_{h}^{2}(Q)},
\end{equation*}
and using \eqref{ine:A},
\begin{align*}
U(y)+J_0(y)\leq&3\tau^{-1}\left\|\theta^{-1/2}e^{\tau\theta\varphi}\partial_{t}y \right\|^{2}_{L^{2}_{h}(Q)}+2\tau^{-1}\left\|\theta^{-1/2}e^{\tau\theta\varphi}g\right\|^{2}_{L_{h}^{2}(Q)}+3J_0(y).
\end{align*}
Hence, by applying the semi-discrete Carleman estimate \eqref{ine:carleman:semi-discrete} to the above inequality, it follows that
\begin{equation*}
\begin{aligned}
     U(y)+J_0(y)\leq &\tilde{C}\left(\left(1+2\tau^{-1} \right)\left\|e^{\tau\theta\varphi}g\right\|^{2}_{L_{h}^{2}(Q)}+\int_{(0,T)\times\omega}\tau^3\theta^3e^{2\tau\theta\varphi}|y|^2\right)\\
     &+\tilde{C}h^{-2}\int_{\mathcal{W}} \left(\Big|y|_{t=0}\Big|^2+\Big|y|_{t=T}\Big|^2\right)e^{2\tau\theta(0)\varphi}.
\end{aligned}
\end{equation*}
Thus, we have the estimate
\begin{equation}\label{ine:new:carleman}
\begin{aligned}
     U(y)+J_0(y)\leq &\overline{C}\left(\left\|e^{\tau\theta\varphi}g\right\|^{2}_{L_{h}^{2}(Q)}+\int_{(0,T)\times\omega}\tau^3\theta^3e^{2\tau\theta\varphi}|y|^2\right)\\
     &+\overline{C}h^{-2}\int_{\mathcal{W}} \left(\Big|y|_{t=0}\Big|^2+\Big|y|_{t=T}\Big|^2\right)e^{2\tau\theta(0)\varphi}. 
\end{aligned}
\end{equation}
In turn, our next task is to compare the terms $\displaystyle \tau^{-1}\left\| \theta^{-1/2}e^{\tau\theta\varphi}\mathcal{A}_{h}y\right\|^{2}_{L_{h}^{2}(Q)}$ and \break $\displaystyle \tau^{-1}\sum_{i,j\in\llbracket 1,d\rrbracket}\int_{Q_{ij}^{\ast}} \theta^{-1}\gamma_{i}\gamma_{j}e^{2\tau\theta\varphi}|D^{2}_{ij}y|^{2}$. To this end, we notice that using the discrete Leibniz rule, the operator $\mathcal{A}_{h}$ can be written as
\begin{align*}
    \mathcal{A}_{h}y&=\sum_{i\in\llbracket 1,d\rrbracket}A_{i}\gamma_{i}\,D_{i}^{2}y+\sum_{i\in\llbracket 1,d\rrbracket}D_{i}\gamma_{i}\,A_{i}D_{i}y\\
&=:\mathcal{A}^{(a)}_{h}y+\mathcal{A}^{(b)}_{h}y.
\end{align*}
Let us compute $\left\| \theta^{-1/2}e^{\tau\theta\varphi}\mathcal{A}_{h}^{(a)}y\right\|^{2}_{L_{h}^{2}(Q)}$. By setting $\alpha_{ij}:= \theta^{-1} e^{2\tau\theta\varphi}A_{i}\gamma_{i}A_{j} \gamma_{j}$ it follows that
    \begin{align}\label{eq:A(a)}
\left\|\theta^{-1/2}e^{\tau\theta\varphi}\mathcal{A}^{(a)}_{h}y\right\|^{2}_{L_{h}^{2}(Q)}=&\int_{Q}\theta^{-1}e^{2\tau\theta\varphi}\left[\sum_{i\in\llbracket 1,d\rrbracket}A_{i}\gamma_{i}\,D_{i}^{2}y \right]\left[\sum_{j\in\llbracket 1,d\rrbracket}A_{j}\gamma_{j}\,D_{j}^{2}y \right] \notag\\
        =&\sum_{i,j\in\llbracket 1,d\rrbracket}\int_{Q}\alpha_{ij} D_{i}^{2}y\,D_{j}^{2}y. 
    \end{align}

In the case $i=j$, thanks to the estimate $(A_{i}\gamma_{i})^{2}=(\gamma_{i})^{2}+\mathcal{O}(h)$ uniformly,
we get
\begin{equation}\label{A:a:i=j}
\begin{aligned}
\left\|\theta^{-1/2}e^{\tau\theta\varphi}\mathcal{A}^{(a)}_{h}y\right\|^{2}_{L_{h}^{2}(Q)}=&\sum_{i\in\llbracket 1,d\rrbracket}\int_{Q}\alpha_{ii} |D_{i}^{2}y|^{2}\\
=&\sum_{i\in\llbracket 1,d\rrbracket}\int_{Q}\theta^{-1} e^{2\tau\theta\varphi}(\gamma_{i})^{2}|D_{i}^{2}y|^{2}+\sum_{i\in\llbracket 1,d\rrbracket}\int_{Q}\theta^{-1} e^{2\tau\theta\varphi}\mathcal{O}(h)|D_{i}^{2}y|^{2}.
\end{aligned}
\end{equation}
Now, for $i\ne j$, an integration by parts with respect to the difference operator $D_{i}$ on (\ref{eq:A(a)}) gives
\begin{equation*}
    \begin{aligned}
        \left\|\theta^{-1/2}e^{\tau\theta\varphi}\mathcal{A}^{(a)}_{h}y\right\|^{2}_{L_{h}^{2}(Q)}=&-\sum_{i,j\in\llbracket 1,d\rrbracket}\int_{Q^{\ast}_{i}} D_{i}y\,D_{i}(\alpha_{ij} D_{j}^{2}y)+\sum_{i,j\in\llbracket 1,d\rrbracket}\int_{\partial_{i}Q}\alpha_{ij} D_{j}^{2}y\,t_{r}^{i}(D_{i}y)\nu_{i}.
\end{aligned}
\end{equation*}
We note that $D_{j}^{2}y=0$ on $\partial_{i}Q$ for $i\ne j$ since $y=0$ on $\partial Q$. Then, the above expression becomes
\begin{equation*}
\begin{aligned}
\left\|\theta^{-1/2}e^{\tau\theta\varphi}\mathcal{A}^{(a)}_{h}y\right\|^{2}_{L_{h}^{2}(Q)}=&-\sum_{i,j\in\llbracket 1,d\rrbracket}\int_{Q^{\ast}_{i}} D_{i}y\,D_{i}\alpha_{ij} \,A_{i}D_{j}^{2}y+D_{i}y\,A_{i}\alpha_{ij} D_{i}D_{j}^{2}y,
           \end{aligned}
\end{equation*}
where we have used the discrete product rule. Analogously, an integration by parts, concerning the difference operator $D_{j}$, yields
\begin{equation*}
    \begin{aligned}
        \left\|\theta^{-1/2}e^{\tau\theta\varphi}\mathcal{A}^{(a)}_{h}y\right\|^{2}_{L_{h}^{2}(Q)}=&\sum_{i,j\in\llbracket 1,d\rrbracket}\int_{Q_{ij}^{\ast}} D_{j}(D_{i}y\,D_{i}\alpha_{ij}) \,A_{i}D_{j}y+D_{j}(D_{i}y\,A_{i}\alpha_{ij}) D_{i}D_{j}y\\
        &-\sum_{i,j\in\llbracket 1,d\rrbracket}\int_{\partial_{j}Q_{i}^{\ast}}D_{i}yD_{i}\alpha_{ij}\, t_{r}^{j}(A_{i}D_{j}y)\nu_{j}+\int_{\partial_{j}Q_{i}^{\ast}}D_{i}yA_{i}\alpha_{ij}\, t_{r}^{j}(D_{i}D_{j}y)\nu_{j}\\
        =&\sum_{i,j\in\llbracket 1,d\rrbracket}\left(\int_{Q_{ij}^{\ast}} D_{j}(D_{i}y\,D_{i}\alpha_{ij}) \,A_{i}D_{j}y+D_{j}(D_{i}y\,A_{i}\alpha_{ij}) D^{2}_{ij}y\right),
           \end{aligned}
\end{equation*}
where we have used $D_{i}y=0$ on $\partial_{j}Q_{i}^{\ast}$ for $i\ne j$. Now, using the discrete Leibniz rule, we get
\begin{equation}
    \begin{aligned}
    \left\|\theta^{-1/2}e^{\tau\theta\varphi}\mathcal{A}_{h}^{(a)}y\right\|^{2}_{L_{h}^{2}(Q)}=&\sum_{i,j\in\llbracket 1,d\rrbracket}\left(\int_{Q_{ij}^{\ast}} D^{2}_{ij}y\,A_{i}D_{i}\alpha_{ij} \,A_{i}D_{j}y+\int_{Q_{ij}^{\ast}} A_{j}D_{i}y\,D^{2}_{ij}\alpha_{ij} \,A_{i}D_{j}y\right)\\
        &+\sum_{i,j\in\llbracket 1,d\rrbracket}\left(\int_{Q_{ij}^{\ast}} |D^{2}_{ij}y|^{2}\,A_{ij}^{2}\alpha_{ij} +A_{j}D_{i}y\,D_{j}A_{i}\alpha_{ij}\, D_{ij}^{2}y\right).
           \end{aligned}
\end{equation}
Moreover, thanks to the Young inequality: $-|ab|\ge-\frac{\tau^{-1/2}}{2}|a|^2-\frac{\tau^{1/2}}{2}|b|^2$,
\begin{equation}\label{ine:A:a}
    \begin{aligned}
    \tau^{-1}\left\|\theta^{-1/2}e^{\tau\theta\varphi}\mathcal{A}_{h}^{(a)}y\right\|^{2}_{L_{h}^{2}(Q)}\geq &-\frac{1}{2}\sum_{i,j\in\llbracket 1,d\rrbracket}\int_{Q_{ij}^{\ast}} \tau^{-3/2}|A_{i}D_{i}\alpha_{ij}|\,|D_{ij}^{2}y|^{2}+\tau^{-1/2}|A_{i}D_{i}\alpha_{ij}|\,|A_{i}D_{j}y|^{2}\\
    &-\frac{1}{2}\sum_{i,j\in\llbracket 1,d\rrbracket}\int_{Q_{ij}^{\ast}} \tau^{-1}|D^{2}_{ij}\alpha_{ij}|\,|A_{j}D_{i}y|^{2}+\int_{Q_{ij}^{\ast}} \tau^{-1}|D^{2}_{ij}\alpha_{ij}| \,|A_{i}D_{j}y|^{2}\\
        &-\frac{1}{2}\sum_{i,j\in\llbracket 1,d\rrbracket}\int_{Q_{ij}^{\ast}} \tau^{-1/2}|D_{j}A_{i}\alpha_{ij}|\,|A_{j}D_{i}y|^{2}+\tau^{-3/2}|D_{j}A_{i}\alpha_{ij}|\,| D_{ij}^{2}y|^{2}\\
        &+\tau^{-1}\sum_{i,j\in\llbracket 1,d\rrbracket}\int_{Q_{ij}^{\ast}} |D^{2}_{ij}y|^{2}\,A_{ij}^{2}\alpha_{ij}.
           \end{aligned}
\end{equation}
Now, by using \eqref{eq:promedioinequality}, $y=0$ on $\partial Q$, and the estimate $e^{-2\tau\theta\varphi}A_{i}D_{i}\alpha_{ij}=\tau\theta^{-1}\partial_{i}\psi \gamma_{i}\gamma_{j}+\mathcal{O}_{\lambda}(sh)+s\mathcal{O}_{\lambda}(sh)$ given in \cite[Theorem 3.5]{AA:perez:2024}, we obtain
\begin{equation*}
\begin{aligned}
  \sum_{i,j\in\llbracket 1,d\rrbracket} \int_{Q_{ij}^{\ast}} |A_{i}D_{i}\alpha_{ij}|\,|A_{i}D_{j}y|^{2}\leq&   \sum_{i,j\in\llbracket 1,d\rrbracket} \int_{Q_{ij}^{\ast}}|A_{i}D_{i}\alpha_{ij}|\,A_{i}(|D_{j}y|^{2})\\
  =&\sum_{i,j\in\llbracket 1,d\rrbracket}\int_{Q_{j}^{\ast}}|A_{i}D_{i}\alpha_{ij}|\,|D_{j}y|^{2}\\
  =&\sum_{j\in\llbracket 1,d\rrbracket}\int_{Q_{j}^{\ast}}\tau\theta^{-1}\gamma_{j} |\nabla\psi|_{\gamma}^{2}\,e^{2\tau\theta\varphi}\,|D_{j}y|^{2}\\
  &+\sum_{j\in\llbracket 1,d\rrbracket}\int_{Q^\ast_j}(\mathcal{O}_{\lambda}(sh)+s\mathcal{O}_{\lambda}(sh))e^{2\tau\theta\varphi}\,|D_{j}y|^{2},
  \end{aligned}
\end{equation*}
where we have used the notation $\displaystyle|\nabla \psi|^{2}_{\gamma}=\sum_{i\in\llbracket 1,d\rrbracket} \gamma_{i}\partial_{i}\psi$. Analogously, thanks to \cite[Theorem 3.5]{AA:perez:2024} we have
\begin{align*}
e^{-2\tau\theta\varphi}A_{i}D_{i}\alpha_{ij}=&\tau \partial_{i}\psi \gamma_{i}\gamma_{j}+\mathcal{O}_{\lambda}(sh)+s\mathcal{O}_{\lambda}(sh),\\
e^{-2\tau\theta\varphi}D_{ij}^{2}\alpha_{ij}=&\tau^{2}\theta\partial_{i}\psi\partial_{j}\psi \gamma_{i}\gamma_{j}+\tau\partial_{ij}^{2}\psi \gamma_{i}\gamma_{j}+s^{2}\mathcal{O}_{\lambda}(sh),\\
e^{-2\tau\theta\varphi}A^{2}_{i}\alpha_{ij}=&\theta^{-1}A_{i}\gamma_{i}A_{j}\gamma_{j}(1+\mathcal{O}_{\lambda}((\tau h)^{2}))=\theta^{-1} \gamma_{i}\gamma_{j}+\mathcal{O}(h)+\mathcal{O}_{\lambda}((sh)^{2}),\\
e^{-2\tau\theta\varphi}D_{j}A_{i}\alpha_{ij}=&\tau \partial_{j}\psi \gamma_{i}\gamma_{j}+\mathcal{O}_{\lambda}(sh)+s\mathcal{O}_{\lambda}(sh)=s\mathcal{O}_{\lambda}(1).
\end{align*}
Thus, by using the above estimates in the remaining terms of the right-hand side in \eqref{ine:A:a}, we obtain the following inequality for the operator $\mathcal{A}_{h}^{(a)}$:
\begin{equation}\label{ine:A:a:2}
    \begin{aligned}
\tau^{-1}\left\|\theta^{-1/2}e^{\tau\theta\varphi}\mathcal{A}_{h}^{(a)}y\right\|^{2}_{L_{h}^{2}(Q)}\geq &\tau^{-1}\sum_{i,j\in\llbracket 1,d\rrbracket}\int_{Q_{ij}^{\ast}} \theta^{-1}\gamma_{i}\gamma_{j}e^{2\tau\theta\varphi}|D^{2}_{ij}y|^{2}-\sum_{i\in\llbracket 1,d\rrbracket}\int_{Q_{i}^{\ast}} \tau\theta|\nabla\psi|^{2}_{\gamma}\partial_{i}\psi\gamma_{i}\,|D_{i}y|^{2}\\
        &-K(y),
           \end{aligned}
\end{equation}
with
\begin{align*}
    K(y):= &\sum_{i,j\in\llbracket 1,d\rrbracket}\int_{Q_{ij}^{\ast}}\left(s^{-1}(\mathcal{O}(h)+\mathcal{O}_{\lambda}((sh)^{2}))+s^{-1/2}\mathcal{O}_{\lambda}(1)\right)e^{2\tau\theta\varphi}\,|D_{ij}^{2}y|^{2}\\
    &+\sum_{j\in\llbracket 1,d\rrbracket}\int_{Q_{j}^{\ast}}\left(\tau^{1/2}\theta^{-1}\gamma_{j} |\nabla\psi|_{\gamma}^{2}+s^{-1/2}(\mathcal{O}_{\lambda}(sh)+s\mathcal{O}_{\lambda}(sh))\right)e^{2\tau\theta\varphi}\,|D_{j}y|^{2}\\
    &+\sum_{i\in\llbracket 1,d\rrbracket}\int_{Q_{i}^{\ast}} \left( \mathcal{O}_{\lambda}(1)+s\mathcal{O}_{\lambda}(sh)\right)e^{2\tau\theta\varphi}\,|D_{i}y|^{2}.
\end{align*}

Finally, for $\mathcal{A}_{h}^{(b)}$, using $D_{i}\gamma_{i}=\mathcal{O}(1)$ and Young's inequality, we have
\begin{align}\label{ine:A:b}
   \left\| \theta^{-1/2}e^{\tau\theta\varphi}\mathcal{A}_{h}^{(b)}y\right\|^{2}_{L_{h}^{2}(Q)} =&\int_{Q}\theta^{-1}e^{2\tau\theta\varphi}\left[\sum_{i\in\llbracket 1,d\rrbracket} D_{i}\gamma_{i}\,A_{i}D_{i}y\right]\left[\sum_{j\in\llbracket 1,d\rrbracket} D_{j}\gamma_{j}\,A_{j}D_{j}y\right]\notag\\
=&\sum_{i,j\in\llbracket 1,d\rrbracket}\int_{Q}\theta^{-1}e^{2\tau\theta\varphi} D_{i}\gamma_{i}\,A_{i}D_{i}y D_{j}\gamma_{j}\,A_{j}D_{j}y\notag\\
\leq & \sum_{i\in\llbracket 1,d\rrbracket}\int_{Q}\theta^{-1}e^{2\tau\theta\varphi} \mathcal{O}(1)\,|A_{i}D_{i}y|^{2}.
\end{align}
Therefore, recalling that $\mathcal{A}_{h}y:= \mathcal{A}_{h}^{(a)}y+\mathcal{A}_{h}^{(b)}y$, combining the estimates \eqref{ine:A:a:2} and \eqref{ine:A:b} yields
\begin{align*}
    \tau^{-1}\left\| \theta^{-1/2}e^{\tau\theta\varphi}\mathcal{A}_{h}y\right\|^{2}_{L_{h}^{2}(Q)}
     &\geq\frac{1}{2}\tau^{-1}\left\| \theta^{-1}e^{\tau\theta\varphi}\mathcal{A}_{h}^{(a)}y\right\|^{2}_{L_{h}^{2}(Q)}-\tau^{-1}\left\|\theta^{-1/2}e^{\tau\theta\varphi}\mathcal{A}_{h}^{(b)}y\right\|^{2}_{L_{h}^{2}(Q)}\\
     &\geq \tau^{-1}\sum_{i,j\in\llbracket 1,d\rrbracket}\int_{Q_{ij}^{\ast}} \theta^{-1}\gamma_{i}\gamma_{j}e^{2\tau\theta\varphi}|D^{2}_{ij}y|^{2}-\sum_{i\in\llbracket 1,d\rrbracket}\int_{Q_{i}^{\ast}} \tau\theta|\nabla\psi|^{2}_{\gamma}\partial_{i}\psi\gamma_{i}\,|D_{i}y|^{2}\\
     &-K(y).
\end{align*}

Hence, thanks to 
\begin{align*}
    U(y)+J_0(y)+K(y)+\sum_{i\in\llbracket 1,d\rrbracket}\tau\theta|\nabla\psi|^{2}\partial_{i}\gamma_{i}e^{2\tau\theta\varphi}|D_{i}y|^{2}\geq &\tau^{-1}\left\| \theta^{-1}e^{\tau\theta\varphi}\partial_{t}y\right\|^{2}_{L_{h}^{2}(Q)}+J_0(y)\\
    &+\tau^{-1}\sum_{i,j\in\llbracket 1,d\rrbracket}\int_{Q^\ast_{ij}}\theta^{-1}\gamma_{i}\gamma_{j}e^{2\tau\theta\varphi}|D_{ij}^{2}y|^{2},
\end{align*}
for $\tau$ large enough, we obtain
\begin{align*}
    U(y)+C\,J_0(y)\geq &I_0(y)+J_0(y),
\end{align*}
where
\begin{align*}
    I_0(y)= &\tau^{-1}\left\| \theta^{-1/2}e^{\tau\theta\varphi}\partial_{t}y\right\|^{2}_{L_{h}^{2}(Q)}+\tau^{-1}\sum_{i,j\in\llbracket 1,d\rrbracket}\int_{Q^\ast_{ij}}\theta^{-1}\gamma_{i}\gamma_{j}e^{2\tau\theta\varphi}|D_{ij}^{2}y|^{2}.
\end{align*}

The last inequality, together with \eqref{ine:new:carleman}, yields the Carleman estimate \eqref{theo:Carleman} for $q=0$.\\
Finally, the Carleman estimate for $q\in\mathbb{R},$ follows from the previous case after a suitable change of variable. In fact, by denoting $u=yv$, with $ v:= (\tau \theta(t))^q$, then  $$L(u):= \partial_{t}u -\sum_{i\in\llbracket 1,d\rrbracket}D_{i}(\gamma_{i}D_{i}u)=vg+yv_t,$$ and applying (\ref{ine:new:carleman}) to $u$, we have
\begin{align}\label{estima:Lema:carleman:01} 
   I_0(yv)+J_0(yv)\leq &C\left(\left\|e^{\tau\theta\varphi}L(u)\right\|^{2}_{L_{h}^{2}(Q)}+\int_{(0,T)\times\omega}\tau^3\theta^3e^{2\tau\theta\varphi}|yv|^2\right)\\
     &+Ch^{-2}\int_{\mathcal{W}} \left(\Big|(yv)|_{t=0}\Big|^2+\Big|(yv)|_{t=T}\Big|^2\right)e^{2\tau\theta(0)\varphi}.\nonumber
\end{align}
Thanks to the inequality $\displaystyle (a+b)^{2}\geq \frac{1}{2}a^{2}-b^{2}$ and noticing that $\theta$ verifies \eqref{eq:theta'} we obtain
\begin{align}\label{estima:Lema:carleman:02}
    \tau^{-1}\left\|\theta^{-1/2}e^{\tau \theta \varphi}\partial_{t}(yv)\right\|^{2}_{L_{h}^{2}(Q)}&= \tau^{-1}\left\|\theta^{-1/2}e^{\tau\theta\varphi}(v\partial_{t}y+y\partial_{t}v) \right\|^{2}_{L_{h}^{2}(Q)}\\
    &\geq \frac{1}{2}\left\|(\tau \theta)^{q-\frac{1}{2}} e^{\tau\theta\varphi}\partial_{t}y\right\|^{2}_{L_{h}^{2}(Q)}-T^2{q}^2\left\|(\tau \theta)^{q-\frac{1}{2}}\theta e^{\tau\theta\varphi}y\right\|^{2}_{L_{h}^{2}(Q)}.\notag
\end{align}
Then, replacing \eqref{estima:Lema:carleman:02} in \eqref{estima:Lema:carleman:01} and using  $L(u)=vg+yv_t$, we get
  \begin{align*}
   \frac{1}{2}I_{2q}(y)-T^2q^2&\left\|(\tau \theta)^{q-\frac{1}{2}}\theta e^{\tau\theta\varphi}y\right\|^{2}_{L_{h}^{2}(Q)}+ J_{2q}(y) 
   \\
   &\le C\left(\left\|(\tau\theta)^{q}e^{\tau\theta\varphi}g\right\|^{2}_{L_{h}^{2}(Q)}+\int_{(0,T)\times\omega}\tau^{3+2q}\theta^{3+2q}e^{2\tau\theta\varphi}|y|^2\right)\\
    &\phantom{\le}+CT^2{q}^2\left\|\tau^{q}\theta^{q+1} e^{\tau\theta\varphi}y\right\|^{2}_{L_{h}^{2}(Q)} \\
     &\phantom{\le}+Ch^{-2}\left(\frac{\tau}{T^{2}\delta}\right)^{2q}\int_{\mathcal{W}} \left(\Big|y|_{t=0}\Big|^2+\Big|y|_{t=T}\Big|^2\right)e^{2\tau\theta(0)\varphi}.
\end{align*}
By increasing the value of the parameter $\tau$, if necessary, the proof follows. 
\end{proof}

\begin{remark}
When the coefficients of the operator $\mathcal{A}_h$ are independent of time, the methodology used to establish the stability of the inverse problem requires only the case $q=0$ in Carleman's inequality. However, in the time-dependent case, $2q=1$ is also necessary; therefore we present \eqref{ine:carleman:semi-discrete} in that general form.
\end{remark}

We end this section with three technical lemmas. The first result, Lemma \ref{lem:T/2}, compares the value of $y$ at $t=T/2$ with the left-hand side of the Carleman estimate \eqref{ine:carleman:semi-discreteNew}. The main difference from the continuous setting is that in this case, there is an additional term at $t=0$ due to the Carleman weight function used in the semi-discrete parabolic operator. The second result, given by Lemma \ref{Appendix-Lema1}, will allow us to absorb the remaining terms in the proof of the stability Theorem \ref{theo:stability}. Finally, Lemma \ref{lem:energy:estimate} provides an energy estimate for the solution of \eqref{system:discrete}.
 
\begin{lemma}\label{lem:T/2} Assume $\tau>1$ is sufficiently large and $y$ is solution of \eqref{system:discrete}. Then, there exists $C>0$ such that, for $q\in\mathbb{R}$, and $t\in[0,T]$,
\begin{multline*}
    \int_{\mathcal{W}}\tau^{2q+1}\theta^{2q+1}\left(t\right)\left|y\left(t,x\right)\right|^2e^{2\tau\theta(t)\varphi(x)}
    \\
    \le C\left(I_{2q}(y)+J_{2q}(y)\right) +\int_{\mathcal{W}} \tau^{2q+1}\theta^{2q+1}(0)\left|y(0,x)\right|^2e^{2\tau\theta(0)\varphi(x)}.
\end{multline*}
\end{lemma}
\begin{proof}
 It suffices to note that, by using $|\theta_t|\le C\theta^2$,
\begin{align*}
\int_0^{t}\partial_t\left(\int_{\mathcal{W}} s^{2q+1}y^2e^{2s\varphi}\right)&
=\int_0^{t}\int_{\mathcal{W}}\left((2s^{2q+1}\tau\partial_{t}\theta\varphi +(2q+1)s^{2q}\tau\partial_{t}\theta)y^2+2s^{2q+1}y\partial_{t}y \right)e^{2s\varphi} \\
&\le C\int_Q(s^{2q+3}+s^{2q+2})y^2e^{2s\varphi}+\int_Q2\left(s^{\frac{2q-1}{2}}|\partial_{t}y |e^{s\varphi}\right)\left(s^{\frac{2q+3}{2}}|y|e^{s\varphi}\right)\\
&\le C\int_Qs^{2q+3}y^2e^{2s\varphi}+\int_Qs^{2q-1}|\partial_{t}y |^2e^{2s\varphi}+\int_Qs^{2q+3}|y|^2e^{2s\varphi},
    \end{align*}
  and the result follows from the definitions of $I_{2q}$ and $J_{2q}$. 
\end{proof}

\begin{corollary}\label{ine:Carleman:semi-discreteNew2}
Under the assumptions of Theorem \ref{theo:Carleman}, there exists $C>0$ such that for $q\in\mathbb{R}$, and $t\in [0,T]$,
\begin{multline*}
    \int_{\mathcal{W}}\tau^{2q+1}\theta^{2q+1}\left(t\right)\left|y\left(t,x\right)\right|^2e^{2\tau\theta(t)\varphi(x)}dx+I_{2q}(y)+J_{2q}(y)
    \\
    \leq  C\left(\int_Qe^{2\tau\theta\varphi}(\tau\theta)^{2q}|g|^2+\int_{(0,T)\times\omega}(\tau\theta)^{2q+3}e^{2\tau\theta\varphi}|y|^2\right) \\
    +Ch^{-2}\int_{\mathcal{W}} (\tau\theta(0))^{2q}\left(\Big|y|_{t=0}\Big|^2+\Big|y|_{t=T}\Big|^2\right)e^{2\tau\theta(0)\varphi}.
\end{multline*}
\end{corollary}
\begin{proof}
   This follows directly from (\ref{ine:carleman:semi-discreteNew}).
\end{proof}

\begin{lemma}\label{Appendix-Lema1}
Assume $\tau_0>1$ is sufficiently large. Then, there exists $C>0$ such that, for $q\in\mathbb{R}$ fixed,    \begin{equation}\label{ine:varphi:T/2}
        \int_{Q}\tau^{2q}\theta^{2q}(t)\left|g\left(\frac{T}{2},x\right)\right|^2e^{2\tau\theta(t)\varphi(x)}\le C\tau^{2q-\frac{1}{2}}\int_{\mathcal{W}} \left|g\left(\frac{T}{2},x\right)\right|^2e^{2\tau\theta\left(\frac{T}{2}\right)\varphi(x)},\quad \forall\tau\geq \tau_0.
        \end{equation}
\end{lemma}
\begin{remark}
    The above estimate is crucial to control certain terms from the right-hand side in the proof of the stability estimate \eqref{theo:stability}. In particular, note that for $2q=1$ we recover the estimate $(3.17)$ in \cite{IY-1998}, which is the estimate that works in that paper.
\end{remark}
\begin{proof}
    First, from \eqref{eq1} and \eqref{eq:theta'} we have $\theta'\left(\frac{T}{2}\right)=0$. Moreover, from (\ref{theta-delta}), for $t\in[0,T]$,
    \[ \theta'(t)=2\left(t-\frac{T}{2}\right)\theta^2(t)=\frac{2\left(t-\frac{T}{2}\right)}{(t+\delta T)^2(T+\delta T-t)^2}, \]
    and
    \begin{align*}
        \theta''(t)=2\theta^2(t)+8\left(t-\frac{T}{2}\right)^2\theta^3(t)\geq 2\theta^2\left(\frac{T}{2}\right).
    \end{align*}
    Since $\delta<\frac{1}{2}$, we obtain $\theta''(t)\geq \frac{2}{T^2}$. Then by integrating twice in time yields \[ \theta(t)\ge \frac{1}{T^2}\left(t-\frac{T}{2}\right)^2+\theta\left(\frac{T}{2}\right). \]

    Namely, from \eqref{eq1}, \eqref{funcion-peso-2} and $\tau>1$ we get
    \begin{align*}
        (\tau-1)\theta(t)\varphi(x)&\le \tau \theta\left(\frac{T}{2}\right)\varphi(x)-\theta\left(\frac{T}{2}\right)\varphi(x)+\frac{\varphi(x)}{T^2} (\tau-1)\left(t-\frac{T}{2}\right)^2,
    \end{align*}
    and therefore \begin{align*}
        s(t)\varphi(x)&\le \theta(t)\varphi(x)+s\left(\frac{T}{2}\right)\varphi(x)+\theta\left(\frac{T}{2}\right)\mu_1-\frac{\mu_0}{T^2} (\tau-1)\left(t-\frac{T}{2}\right)^2,
    \end{align*}
    where $\mu_1:= \sup |\varphi|$ and $\mu_0:= \inf|\varphi|$ are  positive constants.

    Hence
    \begin{align*}
        \int_0^T\theta^{2q}(t)e^{2s(t)\varphi(x)}dt&\le e^{2s\left(\frac{T}{2}\right)\varphi(x)}e^{2\theta\left(\frac{T}{2}\right)\mu_1}\int_0^T\theta^{2q}(t)e^{2\theta(t)\varphi(x)}e^{\left(-2(\tau-1)\frac{\mu_0}{T^2}\left(t-\frac{T}{2}\right)^2\right)}dt \\
        &\le Ce^{2s\left(\frac{T}{2}\right)\varphi(x)}\int_0^T\theta^{2q}(t)e^{-2\theta(t)\mu_0}e^{\left(-2(\tau-1)\frac{\mu_0}{T^2}\left(t-\frac{T}{2}\right)^2\right)}dt
        \\
        &\le Ce^{2s\left(\frac{T}{2}\right)\varphi(x)}\int_0^Te^{\left(-2(\tau-1)\frac{\mu_0}{T^2}\left(t-\frac{T}{2}\right)^2\right)}dt
        \\&
        \le Ce^{2s\left(\frac{T}{2}\right)\varphi(x)}\int_{-\infty}^{+\infty}e^{\left(-2(\tau-1)\frac{\mu_0}{T^2}\mu^2\right)}d\mu\\
        &\le C\frac{Te^{2s\left(\frac{T}{2}\right)\varphi(x)}}{\sqrt{2\mu_0(\tau-1)}}\int_{-\infty}^{+\infty}e^{-\eta^2}d\eta \\
        &\le C\frac{e^{2s\left(\frac{T}{2}\right)\varphi(x)}}{\sqrt{\tau}},
    \end{align*}
which, after multiplying by $\left|g\left(\frac{T}{2},x\right)\right|^2$ and integrating in ${\mathcal{W}}$ proves the Lemma.
\end{proof}
We end this section by proving an energy estimate that will be useful in the next section.
\begin{lemma}\label{lem:energy:estimate}
    Let $y$ be the solution of the system
    \begin{equation}\label{system:semi-discrete}
\begin{cases}
    \partial_{t}y(t,x)-\mathcal{A}_{h}y(t,x)=g(t,x),\quad &(t,x)\in(0,T)\times \mathcal{W},\\
    y(t,x)=0,  &(t,x)\in(0,T)\times\partial\mathcal{W}.
\end{cases}
\end{equation}
Then, for any $T_{0}\in(0,T)$,
\begin{align}\label{ine:energy:T0}
        \int_\mathcal{W} |y|^2(t,x)
        \leq e^{\tilde{C}(t-T_{0})}\left(\int_\mathcal{W} |y|^2(T_{0},x)+\int_{T_{0}}^{t}\int_{\mathcal{W}} |g|^2\right),
\end{align}
for any $t\in (T_{0},T)$, with $\tilde{C}:= \frac{d}{2}\mathrm {reg}(\Gamma)\|b\|^2_\infty+\left\| c\right\|_{\infty}+\frac{1}{2}$, where $\|b\|_\infty^2:= \max\limits_{i\in\{1,...,d\}}\|b_i\|^2_\infty$.
\end{lemma}
\begin{proof} Recalling that 
\begin{equation}
\mathcal{A}_{h}y:= \sum_{i=1}^{d}D_{i}\left(\gamma_{i}(t,x)D_{i}y(t,x)\right)-\sum_{i=1}^{d}b_i(t,x)D_{i}A_{i}y(t,y)-c(t,x)y(t,x),
\end{equation}
by multiplying the main equation of system \eqref{system:semi-discrete} by $y$, integrating over $\mathcal{W}$, and after integration by parts (see \eqref{eq:int:dif}) we have
\begin{align}\label{eq:energy:1}
\frac{\partial}{\partial t}\int_\mathcal{W} \frac{|y|^2}{2}+\sum_{i=1}^d\int_{\mathcal{W}_i^\ast} \gamma_i|D_iy|^2&=\int_\mathcal{W} g y-\sum_{i=1}^d\int_{\mathcal{W}} b_i(A_iD_iy)\,y-\int_{\mathcal{W}}c|y|^2,
\end{align}
where we have used that $y=0$ on the boundary $\partial\mathcal{W}$. Moreover, using that the coefficients $c,b_i$ are bounded, and applying Young's inequality to the right-hand side of \eqref{eq:energy:1} we obtain
        \begin{align*}
        \frac{\partial}{\partial t}\int_\mathcal{W} \frac{|y|^2}{2}+\sum_{i=1}^d\int_{\mathcal{W}_i^\ast} \gamma_i|D_{i}y|^{2}\leq\frac{1}{2}\int_\mathcal{W}|g|^{2}+\sum_{i=1}^d\int_{\mathcal{W}}\frac{\epsilon}{2}\left\| b_{i}\right\|_{\infty}^{2}|A_{i}D_{i}y|^{2}+\int_{\mathcal{W}}\left( \frac{d}{2\epsilon}+\left\| c\right\|_{\infty}+\frac{1}{2}\right)|y|^{2}. 
    \end{align*}
    Let us focus on the integral of the right-hand side with the term $|A_{i}D_{i}y|^{2}$. First, thanks to the inequality \eqref{eq:promedioinequality} and the integration by parts for the average operator \eqref{eq:int:ave} we obtain
    \begin{equation*}
        \frac{\partial}{\partial t}\int_\mathcal{W} \frac{|y|^2}{2}+\sum_{i=1}^d\int_{\mathcal{W}_i^\ast} \gamma_i|D_{i}y|^{2}\leq\frac{1}{2}\int_\mathcal{W}|g|^{2}+\sum_{i=1}^d\int_{\mathcal{W}^{\ast}_{i}}\frac{\epsilon}{2}\left\| b_{i}\right\|_{\infty}^{2}|D_{i}y|^{2}+\int_{\mathcal{W}}\left(\frac{d}{2\epsilon}+\left\| c\right\|_{\infty}+\frac{1}{2}\right)|y|^{2},
    \end{equation*}
since the boundary term is positive. Second, if $\epsilon:= \frac{1}{\mbox {reg}(\Gamma)\|b\|^2_\infty}>0$, it follows 
       \begin{equation*}
        \frac{\partial}{\partial t}\int_\mathcal{W} \frac{|y|^2}{2}\leq\frac{1}{2}\int_\mathcal{W}|g|^{2}+\tilde{C}\int_{\mathcal{W}}|y|^{2},
    \end{equation*}
    with $\tilde{C}:= \frac{d}{2}\mbox {reg}(\Gamma)\|b\|^2_\infty+\left\| c\right\|_{\infty}+\frac{1}{2}$. Finally, multiplying by $e^{-\tilde{C}t}$ the previous inequality we have
        \begin{equation*}
        \frac{\partial}{\partial t}\left( e^{-\tilde{C}t}\int_\mathcal{W} \frac{|y|^2}{2}\right)\leq e^{-\tilde{C}t}\int_\mathcal{W}|g|^{2},
    \end{equation*}
    and the result follows after integrating over the interval $(T_{0},t)$.
\end{proof}
\begin{remark}
    When $b_{i}=0$ for all $i\in\{1,\ldots,d\}$, the inequality \eqref{ine:energy:T0} holds with $\tilde{C}=\|c\|_{\infty}+\frac{1}{2}$.
\end{remark}


\section{An inverse problem for the semi-discrete parabolic operator}\label{sec:inverse:problem}
This section is devoted to the proof of Theorem \ref{theo:stability}, which establishes a stability estimate for the right-hand side $g$ of the system \eqref{system:discrete} in terms of the solution $y$, its derivative $\partial_{t}y$ observed in a subset $\omega$, and the measurement at time $\vartheta=T/2$.

\begin{proof}[Proof of Theorem \ref{theo:stability}] Let $y$ be solution of system \eqref{system:discrete}. Then $z(t,x)=\partial_{t}y(t,x)$ satisfies the following system 
    \begin{align}\label{system:z}
        \begin{cases}
        \partial_{t}z(t,x)-\mathcal{A}_{h}z(t,x)=\mathcal{B}_{h}y(t,x)+\partial_{t}g(t,x),& (t,x)\in(0,T)\times\mathcal{W},\\
        z(t,x)=0, &  x\in (0,T)\times\partial\mathcal{W},\\
        z(T/2,x)=\mathcal{C}_hy(T/2,x)+g(T/2,x),& x\in\mathcal{W},
        \end{cases}
    \end{align}
    where 
\begin{equation*}
\mathcal{A}_{h}z(t,x):= \sum_{i\in\llbracket 1,d\rrbracket}D_{i}\left(\gamma_{i}(t,x)D_{i}z(t,x)\right)-b_{i}(t,x)D_{i}A_{i}z(t,x)-c(t,x)z(t,x),
\end{equation*}
\begin{align*}
\mathcal{B}_{h}y(t,x):= &\sum_{i\in\llbracket 1,d\rrbracket}D_{i}(\partial_{t}\gamma_{i}D_{i}y)-\partial_{t}b_{i}(t,x)D_{i}A_{i}y(t,x)-\partial_{t}c(t,x)y(t,x),\\
\mathcal{C}_{h}y_0(x):= &\sum_{i\in\llbracket 1,d\rrbracket}D_{i}\left(\gamma_{i}\left(\frac{T}{2},x\right)D_{i}y_0(x)\right)-b\left(\frac{T}{2},x\right)D_{i}A_{i}y_0(x)-c\left(\frac{T}{2},x\right)y_0(x),
\end{align*}
and we denote $y_0(x):= y(T/2,x).$
Thanks to the Carleman estimate in Corollary \ref{ine:Carleman:semi-discreteNew2} with $q=0$, and by making $t=T/2$, we get
    \begin{align}\label{ine:2}
        I_{0}(z)+J_{0}(z)+&s\left(T/2\right)\left\|e^{\tau\theta(T/2)\varphi}z|_{t=T/2}\right\|_{L^2_h(\mathcal{W})}^2\\ 
		&\leq  C\left(\int_{Q}e^{2\tau\theta\varphi}(|\partial_{t}g|^{2}+|\mathcal{B}_hy|^2)+\int_{Q_{\omega}}(\tau\theta)^{3}e^{2\tau\theta\varphi}|z|^2\right)\notag\\
		&\phantom{\le}+Ch^{-2}\int_{\Omega} \left(\Big|z|_{t=0}\Big|^2+\Big|z|_{t=T}\Big|^2\right)e^{2\tau\theta(0)\varphi},\notag
    \end{align}
     for any $\tau\geq \tau_{0}(T+T^{2})$, $0<h\leq h_{0}$, $0<\delta \leq 1/2$, $\tau h(\delta T^{2})^{-1}\leq \varepsilon$.
     
  Now, we observe that
  \begin{equation}\label{ine:bound:B}
        |\mathcal{B}_{h}y|\leq \tilde{C}\left(\sum_{i\in\llbracket 1,d\rrbracket}|D_{i}^{2}y|+|D_{i}A_{i}y|+|y|\right),
    \end{equation} From inequality \eqref{ine:Carleman:semi-discreteNew2} with $2q=1$, the solution $y$ of the system \eqref{system:discrete} verifies
  \begin{equation*}
      I_1(y)+J_{1}(y)\leq C\int_{Q}\tau\theta |g|^{2}e^{2s\varphi}+\int_{Q_{ \omega}}\tau^{4}\theta^{4} \varphi^{4}|y|^{2}e^{2s\varphi}+\frac{C}{h^{2}}\int_{\mathcal{W}}\tau\theta(0)\left(\Big|y|_{t=0}\Big|^{2}+\Big|y|_{r=T}\Big|^{2}\right)e^{2\tau\theta(0)\varphi}.
  \end{equation*}
  
Thus, substituting the above estimate into the right-hand side of \eqref{ine:2} and, if necessary, increasing the parameter $\tau$, we obtain
\begin{align}
I_{0}(z)+J_{0}(z)+&s\left(T/2\right)\left\|e^{\tau\theta(T/2)\varphi}z|_{t=T/2}\right\|_{L^2_h(\mathcal{W})}^2\notag\\
\leq& C\left(\int_{Q}\left[|\partial_{t}g|^{2}+s|g|^{2}\right]e^{2s\varphi} \right)+C\int_{Q_{\omega}}s^{3}|z|^{2}e^{2s\varphi}+C\int_{Q_{\omega}}s^{4}|y|^{2}e^{2s\varphi}\notag\\
&+Ch^{-2}\int_{\mathcal{W}}\left(\Big|z|_{t=0}\Big|^{2}+\Big|z|_{t=T}\Big|^{2}\right)e^{2\tau\theta(0)\varphi}\label{ine:z:y}\\
&+\frac{C\tau\theta(0)}{h^{2}}\int_{\mathcal{W}}\left(\Big|y|_{t=0}\Big|^{2}+\Big|y\Big|_{t=T}|^{2}\right)e^{2\tau\theta(0)\varphi}\notag.
\end{align}
Moreover, using the assumption \eqref{assump:g} it follows that there exists a constant $C>0$ such that
     \begin{equation*}
\int_{Q}\left(|\partial_{t}g|^{2}+s|g|^{2}\right)e^{2s\varphi}\leq C\int_{Q}s\Big|g|_{t=T/2}\Big|^{2}e^{2s\varphi}\, \mbox{ for all }(t,x)\in Q \mbox{ and }\tau\geq \tau_{0}.
     \end{equation*}
Thus, using the above estimate in \eqref{ine:z:y} 
    and from Lemma \ref{Appendix-Lema1} with $2q=1$, we get
    \begin{equation}\label{ine:y:z:2}
        \begin{aligned}
         I_{0}(z)+J_{0}(z)+&s\left(T/2\right)\left\|e^{\tau\theta(T/2)\varphi}z|_{t=T/2}\right\|_{L^2_h(\mathcal{W})}^2\\
         \leq&   C\sqrt{\tau}\left(\int_{{\mathcal{W}}}\left|g\left( \frac{T}{2},x\right)\right|^{2}e^{2\tau\theta(T/2)\varphi(x)} \right)+C\int_{Q_{\omega}}s^{3}e^{2s\varphi}|z|^{2}+C\int_{Q_{ \omega}}s^{4}|y|^{2}e^{2s\varphi}\\
&+Ch^{-2}\int_{\mathcal{W}}\left(\Big|z|_{t=0}\Big|^{2}+\Big|z|_{t=T}\Big|^{2}\right)e^{2\tau\theta(0)\varphi}\\
&+C\tau\theta(0)h^{-2}\int_{\mathcal{W}}\left(\Big|y|_{t=0}\Big|^{2}+\Big|y|_{t=T}\Big|^{2}\right)e^{2\tau\theta(0)\varphi}.
        \end{aligned}
    \end{equation}
\par On the other hand, recalling that $z(T/2,x)=\mathcal{C}_{h}y_{0}(x)+g(T/2,x)$ 
 and by the definition of $\mathcal{C}_{h}$ we get 
  \begin{equation}\label{ine:almost}
  \begin{aligned}
    \left\|e^{\tau\theta(T/2)\varphi}z|_{t=T/2}\right\|_{L^2_h(\mathcal{W})}^2 \geq &-C\int_{\mathcal{W}}|\mathcal{D}y_{0}|^2\,e^{2\tau\theta(T/2)\varphi(x)}\\
   &+C\int_{\mathcal{W}}\left| g\left( \frac{T}{2},x\right)\right|^{2}e^{2\tau\theta(T/2)\varphi(x)},
   \end{aligned}
  \end{equation}
  where $|\mathcal{D} y_{0}|^2:= \sum_{i\in\llbracket 1,d\rrbracket} |D_{i}^{2}y_{0}|^{2}+|D_{i}A_{i}y_{0}|^{2}+|y_{0}|^{2}$.

Combining \eqref{ine:y:z:2}, \eqref{ine:almost},
and increasing $\tau$ if necessary, we can absorb the term \break $\|g(T/2)e^{\tau\theta(T/2)\varphi}\|^{2}_{L^{2}_{h}(\mathcal{W})}$ from the right-hand side to obtain 
\begin{equation}\label{ine:g}
\begin{aligned}
    s(T/2)\|e^{\tau\theta(T/2)\varphi}g|_{t=T/2}\|^{2}_{L^{2}_{h}(\mathcal{W})}\leq &Cs(T/2)\int_{\mathcal{W}}|\mathcal{D}y_{0}|^2\,e^{2\tau\theta(T/2)\varphi(x)}\\
    &+C\int_{Q_{\omega}}s^{3}e^{2s\varphi}|z|^{2}+C\int_{Q_{ \omega}}s^{4}|y|^{2}e^{2s\varphi}\\
&+Ch^{-2}\int_{\mathcal{W}}\left(\Big|z|
_{t=0}\Big|^2+\Big|z|_{t=T}\Big|^{2}\right)e^{2\tau\theta(0)\varphi}\\
&+C\tau\theta(0)h^{-2}\int_{\mathcal{W}}\left(\Big|y|_{t=0}\Big|^{2}+\Big|y|_{t=T}\Big|^{2}\right)e^{2\tau\theta(0)\varphi}.
\end{aligned}
\end{equation}
Note that 
\begin{equation}\label{eq:exp:bound}
    \exp\left( 2\tau\theta(0)\varphi(x)\right)=\exp\left( 2\tau\theta(T)\varphi(x)\right)\leq \exp\left(\frac{-C\tau}{\delta T^{2}} \right),
\end{equation} since $\theta(0)=\theta(T)\leq (\delta T^{2})^{-1}$ and $\sup \varphi <0$. Analogously, we have
\begin{equation}\label{eq:exp:bound:2}
\exp\left(2\tau\theta(T/2)\varphi\right)\geq \exp\left(-C'\frac{\tau}{T^{2}}\right),
\end{equation} 
where we have used that $\varphi(x) <0$ and $\theta(T/2)=\frac{4}{T^{2}(1+2\delta)^{2}}\leq \frac{4}{T^{2}}$. Thus, by using \eqref{eq:exp:bound} to estimate terms on the right-hand side of \eqref{ine:g}, and \eqref{eq:exp:bound:2} for the left-hand side, we arrive to
\begin{equation}\label{ine:g:2}
\begin{aligned}
    \tau\int_{\mathcal{W}}\left|g\left( \frac{T}{2},x\right)\right|^{2}\leq &C\tau e^{\frac{C''\tau}{T^{2}}}\left\| y_{0}\right\|^{2}_{H_h^{2}(\mathcal{W})}+Ce^{\frac{C''\tau}{T^{2}}}\int_{Q_{\omega}}s^{3}e^{2s\varphi}|z|^{2}+Ce^{\frac{C''\tau}{T^{2}}}\int_{Q_{ \omega}}s^{4}|y|^{2}e^{2s\varphi}\\
&+Ch^{-2}e^{-\frac{C''\tau}{\delta T^{2}}}\left(\|z|_{t=0}\|_{L^2_h(\mathcal{W})}^2+\|z|_{t=T}\|_{L^2_h(\mathcal{W})}^2\right)\\
&+C\tau\theta(0)h^{-2}e^{\frac{-C''\tau}{\delta T^{2}}}\left(\|y|_{t=0}\|_{L^2_h(\mathcal{W})}^2+\|y|_{t=T}\|_{L^2_h(\mathcal{W})}^2\right).
\end{aligned}
\end{equation}

Finally, applying Lemma \ref{lem:energy:estimate} to \eqref{ine:g:2} for the solutions of systems \eqref{system:z} and \eqref{system:discrete}, respectively, yields
\begin{equation}\label{ine:g:3}
\begin{aligned}
    \tau\int_{\mathcal{W}}\left|g\left( \frac{T}{2},x\right)\right|^{2}\leq &C\tau e^{\frac{C''\tau}{T^{2}}}\left\| y_{0}\right\|^{2}_{H^{2}_{h}(\mathcal{W})}+Ce^{\frac{C''\tau}{T^{2}}}\int_{Q_{\omega}}s^{3}e^{2s\varphi}|z|^{2}+Ce^{\frac{C''\tau}{T^{2}}}\int_{Q_{ \omega}}s^{4}|y|^{2}e^{2s\varphi}\\
&+Ch^{-2}e^{-\frac{C''\tau}{\delta T^{2}}}\left(\|z|_{t=0}\|_{L^2_h(\mathcal{W})}^2+\int_{0}^{T}\int_{\mathcal{W}}|\mathcal{B}_{h}y_{0}+\partial_{t}g|^{2}\right)\\
&+C\tau\theta(0)h^{-2}e^{\frac{-C''\tau}{\delta T^{2}}}\left(\|y|_{t=0}\|_{L^2_h(\mathcal{W})}^2+\int_{0}^{T}\int_{\mathcal{W}}|g|^{2}\right).
\end{aligned}
\end{equation}

Let $\tau_1>0$ be such that $\tau\geq \tau_1$. Then $e^{-\frac{C''\tau}{\delta T^{2}}}\leq e^{-\frac{C''\tau_1}{\delta T^{2}}}$.Choosing $\delta$ small enough so that $\frac{\tau_1}{T^{2}\delta}=\frac{\varepsilon_{0}}{h}$, we obtain
\begin{equation*}
\begin{aligned}
    \tau\int_{\mathcal{W}}\left|g\left( \frac{T}{2},x\right)\right|^{2}\leq &C\tau e^{\frac{C''\tau}{T^{2}}}\left\| y_{0}\right\|^{2}_{H^{2}_{h}(\mathcal{W})}+Ce^{\frac{C''\tau}{T^{2}}}\int_{Q_{\omega}}s^{3}e^{2s\varphi}|z|^{2}+Ce^{\frac{C''\tau}{T^{2}}}\int_{Q_{ \omega}}s^{4}|y|^{2}e^{2s\varphi}\\
&+Ce^{-\frac{C''}{h}}\left(\|z|_{t=0}\|_{L^2_h(\mathcal{W})}^2+\int_{0}^{T}\int_{\mathcal{W}}|\mathcal{B}_{h}y_{0}+\partial_{t}g|^{2}\right)\\
&+Ce^{\frac{-C''}{h}}\left(\|y|_{t=0}\|_{L^2_h(\mathcal{W})}^2+\int_{0}^{T}\int_{\mathcal{W}}|g|^{2}\right).
\end{aligned}
\end{equation*}

Finally, by using assumption \eqref{assump:g}, regarding the definition of $\mathcal{B}_{h}y_{0}$, and increasing $\tau$ if necessary, it follows that

\begin{equation}\label{ine:g:5}
\begin{aligned}
    \tau\int_{\mathcal{W}}\left|g\left( \frac{T}{2},x\right)\right|^{2}\leq &C \tau e^{\frac{C''\tau}{T^{2}}}\left\| y_{0}\right\|^{2}_{H^{2}_{h}(\mathcal{W})}+Ce^{\frac{C''\tau}{T^2}}\int_{Q_{\omega}}s^{3}e^{2s\varphi}|z|^{2}+Ce^{\frac{C''\tau}{T^2}}\int_{Q_{ \omega}}s^{4}|y|^{2}e^{2s\varphi}\\
&+Ce^{-\frac{C''}{h}}\left(\|y|_{t=0}\|_{L^2_h(\mathcal{W})}^2+\|z|_{t=0}\|_{L^2_h(\mathcal{W})}^2\right).
\end{aligned}
\end{equation}
Notice that condition \eqref{assump:g} and mean value Theorem imply that there exists a constant $C'>0$ such that $|g(t,x)|\leq C'\left| g(T/2,x)\right|$ for all $(t,x)\in Q$. Substituting this last inequality in \eqref{ine:g:5} the proof is concluded.
\end{proof}

From the proof of Theorem \ref{theo:stability}, we observe that if the coefficients $\gamma_i, b_i$ and $c$ are independent of time, the operator $\mathcal{B}_h=0$. Thus, we have the following.
\begin{corollary}\label{theo:stability_idependentTime}
Let $\gamma_i, b_i$, $i=1,\ldots,d$, and $c$ be independent of time, $\psi$ that satisfies \eqref{assumtion:psi} and $\varphi$ according to \eqref{funcion-peso-2}. Let $g$ satisfy \eqref{assump:g} and let $y$ be the solution of the system \eqref{system:discrete}. Then, there exist positive constants $C$, $C''$,  $s_{0}\geq 1$, $h_{0}>0$, $\varepsilon >0$, depending on $\omega$, $\omega_{0}$, $\mbox{reg}^{0}$, $T$, such that for any $\tau\geq \tau_{0}(T+T^{2})$, $0<h\leq h_{0}$, there exists $0<\delta(h)\leq1/2$, with $\tau h(\delta T^{2})^{-1}\leq \varepsilon$, and the estimate
\begin{equation*}
    \begin{aligned}
    \|g\|_{L_h^2(\mathcal{W})}\leq &C\left(e^{\frac{C''}{T^2}\tau}\|y|_{t=\vartheta}\|_{H_h^2(\mathcal{W})}+e^{\frac{C''}{T^2}\tau}\|e^{s\alpha}\partial_{t}y\|_{L_h^2(Q_\omega)}
    +e^{-\frac{C''}{h}}\|\partial_{t}y|_{t=0}\|_{L^2_h(\mathcal{W})}\right),
    \end{aligned}
\end{equation*}
holds for $y\in\mathcal{C}^{1}([0,T],\overline{\mathcal{W}})$ and $Q_{\omega}:= (0,T)\times\omega$.
\end{corollary}

The steps to prove Corollary \ref{theo:stability_idependentTime} are similar to those in the previous proof of Theorem \ref{theo:stability}. The main difference in the time-dependent case is the estimate for the operator ${\mathcal{B}}_{h}$, since it does not involve a second-order operator of $y$. In that sense, the proof of Corollary \ref{theo:stability_idependentTime} requires only the case $q=0$ from Theorem \ref{theo:Carleman}, and it is not necessary to use Lemma \ref{Appendix-Lema1}.

\subsection{Stability for the coefficient inverse problem}
An inverse problem related to the one described above is that when the source term has the form $g(t,x)=f(x)R(t,x)$. In this case, the aim is to estimate $f$ from the observations of $y$, the solution of
\begin{equation}\label{system:coefficient}
\begin{cases}
    \partial_{t}y(t,x)-\mathcal{A}_{h}y(t,x)=f(x)R(t,x),\quad &(t,x)\in Q,\\
    y(t,x)=0,& (t,x)\in (0,T)\times \partial\mathcal{W},\\
    y(0,x)=y_{ini}(x), &x\in\mathcal{W}.
\end{cases}
\end{equation}
 Indeed, assuming $R\in C^1([0,T];\mathcal{W})$ and that there exists a positive constant $\alpha>0$ such that  $$|R(\vartheta,x)|\geq\alpha,\quad\forall x\in\mathcal{W},$$
we have that $g(t,x):= f(x)R(t,x)$, for $f\in L^\infty_{h}({\mathcal{W}})$, verifies condition \eqref{assump:g}. Thus, by applying Theorem \ref{theo:stability} we have
\begin{equation}
        \begin{aligned}
		    \|f\|_{L_h^2(\mathcal{W})}\leq &Ce^{\frac{C''}{T^2}\tau}\left(\|y|_{t=\vartheta}\|_{H_h^2(\mathcal{W})}+\|e^{s\alpha}\partial_{t}y\|_{L_h^2(Q_\omega)}+\|e^{s\alpha}y\|_{L_h^2(Q_\omega)}\right)\\
            &+Ce^{-\frac{C''}{h}}\left(\|y|_{t=0}\|_{L^2_h(\mathcal{W})}+\|\partial_{t}y|_{t=0}\|_{L^2_h(\mathcal{W})}\right).
            \end{aligned}
		\end{equation} 
When the operator $\mathcal{A}_h$ is independent of time, using Corollary \ref{theo:stability_idependentTime} we obtain the following
        \begin{equation}\label{stability:02}
        \begin{aligned}
		    \|f\|_{L_h^2(\mathcal{W})}\leq &C\left(e^{\frac{C''}{T^2}\tau}\|y|_{t=\vartheta}\|_{H_h^2(\mathcal{W})}+e^{\frac{C''}{T^2}\tau}\|e^{s\alpha}\partial_{t}y\|_{L_h^2(Q_\omega)}+e^{-\frac{C''}{h}}\|\partial_{t}y|_{t=0}\|_{L^2_h(\mathcal{W})}\right).
            \end{aligned}
		\end{equation} 

\section{Stability and reconstruction of a coefficient inverse problem}\label{sec:stability:reconstruction}

In this section, we are interested in the determination of a time-independent coefficient of zero-order, $p$, in \eqref{system:coefficient:02}. We consider the case in which the coefficients $\gamma_i, b_i$, $i=1,\ldots,d$, and $c$ are independent of time. We first establish the stability result for the coefficient inverse problem, which is a consequence of Corollary \ref{theo:stability_idependentTime}.
 \begin{theorem}[Stability for the coefficient inverse problem]\label{theorem:convergence}
 Let $\psi$ satisfy \eqref{assumtion:psi} and $\varphi$ be given by \eqref{funcion-peso-2}. Assume that $\gamma_i, b_i$, $i=1,\ldots,d$, and $c$ are independent of time. Let us consider $y_{p_1}$ and $y_{p_2}$ solutions of \eqref{system:coefficient:02} associated to $p_1,p_2\in \mathcal{X}_{m}$, respectively, such that there exists $\alpha>0$ with $|y_{p_1}(T/2,\cdot)|>\alpha$. Then, there exist positive constants $C$, $C''$,  $\tau_{0}\geq 1$, $h_{0}>0$, $\varepsilon >0$, depending on $\omega$, $\omega_{0}$, $T$, such that for any $p_2\in \mathcal{X}_{m}$,
\begin{equation}\label{stability:03}
    \begin{aligned}
		\|p_1-p_2\|_{L_h^2(\mathcal{W})}\leq &C
        \||\Lambda_{p_1}-\Lambda_{p_2}\||,
    \end{aligned}
\end{equation} 
where 
        $$
        \begin{aligned}
            \||\Lambda_{p_1}-\Lambda_{p_2}\||:= &
            e^{\frac{C''}{T^2}\tau}\|y_{p_1}|_{t=T/2}-y_{p_2}|_{t=T/2}\|_{H_h^2(\mathcal{W})}\\
            &+e^{\frac{C''}{T^2}\tau}\|e^{s\alpha}(\partial_{t}y_{p_1}-\partial_{t}y_{p_2})\|_{L_h^2(Q_\omega)}\\
            &+e^{-\frac{C''}{h}}\|\partial_{t}y_{p_1}|_{t=0}-\partial_{t}y_{p_2}|_{t=0}\|_{L^2_h(\mathcal{W})}.
        \end{aligned}$$
 \end{theorem}
\begin{proof}
     Define $z=y_{p_1}-y_{p_2}$. Then $z$ satisfies the system
    \begin{equation*}
        \begin{cases}
            \partial_{t}z(t,x)-\mathcal{A}_{h}z(t,x)+p_2(x)z(t,x)=(p_2(x)-p_1(x))y_{p_1}(t,x),\quad &(t,x)\in Q,\\
            z(0,x)=0, &x\in\mathcal{W},\\
            z(t,x)=0, &(t,x)\in(0,T)\times\partial\mathcal{W}.
        \end{cases}
    \end{equation*}
    Setting $g(t,x)=(p_2(x)-p_1(x))y_{p_1}(t,x)$, we observe that $$|\partial_t g(t,x)|\leq |p_2(x)-p_1(x)| |\partial_t y_{p_1}(t,x)|\frac{|y_{p_1}(T/2,x)|}{\alpha}\leq \frac{C}{\alpha}|g(T/2,x)|.$$
    
    Applying Corollary \ref{theo:stability_idependentTime} to $z$, there exist positive constants $C$, $C''$,  $s_{0}\geq 1$, $h_{0}>0$, $\varepsilon >0$, depending on $\omega$, $\omega_{0}$, $\mbox{reg}^{0}$, $T$, such that for any $\tau\geq \tau_{0}(T+T^{2})$, $0<h\leq h_{0}$, there exists $0<\delta(h)\leq 1/2$, with $\tau h(\delta T^{2})^{-1}\leq \varepsilon$, and verifying
\begin{equation*}
    \| (p_2-p_1)y_{p_1}\|_{L_h^2(\mathcal{W})}\leq C\left(e^{\frac{C''}{T^2}\tau}\|z|_{t=T/2}\|_{H_h^2(\mathcal{W})}+e^{\frac{C''}{T^2}\tau}\|e^{s\alpha}\partial_{t}z\|_{L_h^2(Q_\omega)}
    +e^{-\frac{C''}{h}}\|\partial_{t}z|_{t=0}\|_{L^2_h(\mathcal{W})}\right).
\end{equation*} 
We conclude with $\alpha \| (p_2-p_1)\|_{L_h^2(\mathcal{W})}\leq \| (p_2-p_1)y_{p_1}\|_{L_h^2(\mathcal{W})}$.
\end{proof}
\begin{remark}
The result of Theorem \ref{stability:03} remains valid even if the initial conditions associated to $y_{p_1}$ and $y_{p_2}$ are different. This follows from the facts that, when $\mathcal{A}_h$ is time-independent, the initial conditions do not appear in the stability result presented in Corollary \ref{theo:stability_idependentTime}. Similarly, in the continuous case, the initial condition does not appear in the corresponding stability inequality (see, for instance, equation (2.3) in \cite{Yamamoto_2001}).
\end{remark}
We now present the Carleman inequality required for the reconstruction algorithm \eqref{Alporithmo_Reconstruccion}.
\begin{theorem}[Carleman estimate]\label{theo:Carleman:02}
Assume that $\psi$ satisfies \eqref{assumtion:psi} and $\varphi$ is given by \eqref{funcion-peso-2}. For $\lambda\geq 1$ sufficiently large, there exist $C$, $\tau_{0}\geq 1$, $h_{0}>0$, $\varepsilon >0$, depending on $\omega$, $\omega_{0}$, $T$, $m$, $\lambda$, such that    \begin{align}\label{ine:carleman:semi-discreteNew:02}
		\int_{\mathcal{W}}\tau e^{2\tau\theta(T/2)\varphi}\Big|u|_{t=T/2}\Big|^2+I_{0}(u)+J_{0}(u) 
		\leq&  C\left(\int_Q\rho_a|L_p(u)|^2+\int_{Q_\omega}\rho_b|u|^2 +\int_{\mathcal{W}}\rho_c\Big|u|_{t=0}\Big|^2\right),
	\end{align}
where the weight functions are given by 
\begin{align*}           \rho_a(t,x):= &e^{2\tau\theta(t)\varphi(x)}+h^{-2}e^{-2\theta(0)\tau\inf|\varphi|},\\
\rho_b(t,x):= &(\tau\theta(t))^3e^{2\tau\theta(t)\varphi(x)},\\ 
\rho_c(t,x):= &(h^{-2}+\tau\theta(0))e^{2\tau\theta(0)\varphi(x)}+e^{-2\tau\theta(0)\inf|\varphi|},
        \end{align*}
and $$L_p(u):= \partial_tu-\mathcal{A}_h u+pu,$$
for all $p\in \mathcal{X}_{m}$, $\tau\geq \tau_{0}(T+T^{2})$, $0<h\leq h_{0}$, $0<\delta \leq 1/2$, $\tau h(\delta T^{2})^{-1}\leq \varepsilon$, and $u\in \mathcal{C}^{1}([0,T],\overline{\mathcal{W}})$.
\end{theorem}

\begin{proof}
    The result follows from Theorem \ref{theo:Carleman} and Lemmas \ref{lem:T/2} and \ref{lem:energy:estimate}. Indeed, using Lemma \ref{lem:T/2} with $q=0$, there exists a positive constant $C>0$ such that 
    \begin{equation*}
        \int_{\mathcal{W}}\tau \theta e^{2\tau\theta(T/2)\varphi}\Big|u|_{t=T/2}\Big|^2+I_{0}(u)+J_{0}(u) \leq C \left( I_{0}(u)+J_{0}(u) \right)+\int_{\mathcal{W}}\tau \theta(0) e^{2\tau\theta(0)\varphi}\Big|u|_{t=0}\Big|^2,
    \end{equation*}
    and using the Carleman inequality from Theorem \ref{theo:Carleman}, we obtain
    \begin{align}\label{eq:Teo:01}
		\int_{\mathcal{W}}\tau \theta e^{2\tau\theta(T/2)\varphi}\Big|u|_{t=T/2}\Big|^2+I_{0}(u)&+J_{0}(u)\notag \\&
		\leq  C\left(\int_Qe^{2\tau\theta\varphi}|L_p(u)|^2+\int_{(0,T)\times\omega}(\tau\theta)^{3}e^{2\tau\theta\varphi}|u|^2\right)\\
		&+Ch^{-2}\int_{\mathcal{W}} \left(\Big|u|_{t=0}\Big|^2+\Big|u|_{t=T}\Big|^2\right)e^{2\tau\theta(0)\varphi}\notag\\
        &+\int_{\mathcal{W}}\tau \theta(0) e^{2\tau\theta(0)\varphi}\Big|u|_{t=0}\Big|^2.\notag
	\end{align}
    Finally, using Lemma \ref{lem:energy:estimate}, we have 
    \begin{align*}
        \int_\mathcal{W} \Big|u|_{t=T}\Big|^2
        \leq C\left(\int_\mathcal{W} \Big|u|_{t=0}\Big|^2+\int_{Q} |L_p(u)|^2\right),
    \end{align*}
    and replacing in \eqref{eq:Teo:01}, and using the definition of $\rho_a,$ $\rho_b$ and $\rho_c$, we conclude the proof.
\end{proof} 
Before going further, for any $p\in\mathcal{X}_{m}$, we introduce the space of the trajectories, 
 $$\mathcal{V}_p:= \{ z\in L^2(0,T;H_h^1(\mathcal{W})):\; L_p(z)\in L^2((0,T)\times\mathcal{W}), \; z|_{\partial\mathcal{W}}=0,\;\textrm{and } z|_{t=0}\in L_{h}^2(\mathcal{W})\},$$
endowed with the norm 
$$\|z\|_{\mathcal{V}_p,\tau}:= \left(\int_Q\rho_a|L_p(z)|^2+\int_{Q_\omega}\rho_b|z|^2 +\int_{\mathcal{W}}\rho_c\Big|z|_{t=0}\Big|^2\right)^{1/2},$$
which thanks to Theorem \ref{theo:Carleman:02} is a norm in $\mathcal{V}_p$, for $\tau$ sufficiently large .

 Now, for any $\mu\in L_{h}^2(Q_\omega)$ and $\nu\in L_{h}^2(\mathcal{W})$, we introduce the functional $\mathcal{J}_{\tau,p}[\mu,\nu]:\mathcal{V}_p\to \mathbb{R}$, given by
\begin{equation}\label{FunctionalJ}
            \mathcal{J}_{\tau,p}[\mu,\nu](u):= \frac{1}{2}\int_Q\rho_a|L_p(u)|^2+\frac{1}{2}\int_{Q_\omega}\rho_b|u-\mu|^2 +\frac{1}{2}\int_{\mathcal{W}}\rho_c\Big|u|_{t=0}-\nu\Big|^2.
\end{equation}
       
 \begin{theorem}
    Assume that \eqref{assumtion:psi} and \eqref{funcion-peso-2} hold, and that $\mu\in L_{h}^2(Q_\omega)$ and $\nu\in L_{h}^2(\mathcal{W})$. Then, for $\lambda\geq 1$ sufficiently large, there exist $\tau_{0}\geq 1$, $h_{0}>0$, $\varepsilon >0$, depending on $\omega$, $\omega_{0}$, $T$, $m$, $\lambda$, such that the functional $\mathcal{J}_{\tau,p}[\mu,\nu]$ defined in \eqref{FunctionalJ} is continuous, strictly convex and coercive on $(\mathcal{V}_p,\|\cdot\|_{\mathcal{V}_p,\tau})$; hence it admits a unique minimizer $u^\ast_p$ in $\mathcal{V}_p$, for any $p\in \mathcal{X}_{m}$, $\tau\geq \tau_{0}(T+T^{2})$, $0<h\leq h_{0}$, $0<\delta \leq 1/2$, $\tau h(\delta T^{2})^{-1}\leq \varepsilon$.
    
Moreover, for any data $\mu\in L_{h}^{2}(Q_\omega)$ and $\nu\in L_{h}^{2}(\mathcal{W})$, the minimizer $u^\ast_p$ of $\mathcal{J}_{\tau,p}[\mu,\nu]$ satisfies:
$$\|u^\ast_p\|_{\mathcal{V}_p,\tau}\leq 2 \left(\int_{Q_\omega}\rho_b|\mu|^2 +\int_{\mathcal{W}}\rho_c|\nu|^2\right)^{1/2}.$$
 \end{theorem}
 \begin{proof}
     The proof follows from the  decomposition 
     \begin{equation}\label{eq:teo:02}
         \mathcal{J}_{\tau,p}[\mu,\nu](u)=\mathcal{J}_{\tau,p}[0,0](u)+\mathcal{J}_{\tau,p}[\mu,\nu](0)-\int_{Q_\omega}\rho_b\,\mu\,u -\int_{\mathcal{W}}\rho_c\,\nu\,u|_{t=0},
     \end{equation}
     where we observe that $\mathcal{J}_{\tau,p}[0,0](u)=\frac{1}{2}\|u\|^2_{\mathcal{V}_p,\tau}$. Thus,  $\mathcal{J}_{\tau,p}[\mu,\nu](u)$ is the sum of a strictly convex function (the norm of $\mathcal{V}_p$) and a linear continuous operator. Namely, $\mathcal{J}_{\tau,p}[\mu,\nu]$ defined in \eqref{FunctionalJ} is continuous, strictly convex and coercive on $(\mathcal{V}_p,\|\cdot\|_{\mathcal{V}_p,\tau})$. Therefore $\mathcal{J}_{\tau,p}[\mu,\nu]$ admits a unique minimizer $u^\ast_p$ in $\mathcal{V}_p$, for any $p\in \mathcal{X}_{m}$, $\tau\geq \tau_{0}(T+T^{2})$, $0<h\leq h_{0}$, $0<\delta \leq 1/2$, $\tau h(\delta T^{2})^{-1}\leq \varepsilon$.
     
     Finally, denoting by $u^\ast_{p}$ the minimizer we have 
     $$\mathcal{J}_{\tau,p}[\mu,\nu](u)\leq \mathcal{J}_{\tau,p}[\mu,\nu](0),$$
     and using \eqref{eq:teo:02}, we obtain 
     $$\frac{1}{2}\|u^\ast_p\|^2_{\mathcal{V}_p,\tau}+\mathcal{J}_{\tau,p}[\mu,\nu](0)-\int_{Q_\omega}\rho_b\,\mu\,u^\ast_p -\int_{\mathcal{W}}\rho_c\,\nu\,u^\ast_p|_{t=0}\leq \mathcal{J}_{\tau,p}[\mu,\nu](0).$$
     Using the inequality $2ab\leq 2a^2+\frac{b^2}{2}$, the proof follows.
 \end{proof}

\begin{proof}[Proof of Theorem \ref{theo:convergence}(Convergence of Algorithm \eqref{Alporithmo_Reconstruccion})]
    Let $p^\ast\in \mathcal{X}_{m}$ such that there exists $\alpha>0$ with
$$|y_{p^{\ast}}(T/2,\cdot)|>\alpha,$$ and $y_{p^\ast}$ is the solution of \eqref{system:coefficient:02} with $p=p^\ast$. We assume the data 
$$\Lambda_{p^\ast}:= (y_{p^\ast}|_{t=T/2},\partial_ty_{p^\ast}|_{Q_{\omega}},\partial_ty_{p^\ast}|_{\{t=0\}\times\mathcal{W}})$$
are known.

Given $p_k\in \mathcal{X}_{m}$ we consider $y_{p_k}$, the solution of \eqref{system:coefficient:02}, with $p=p_k$. Let $\mu_k:= \partial_ty_{p_k}-\partial_ty_{p*}$ on $Q_\omega$ and $\eta_k:= \partial_ty_{p_k}|_{t=0}-\partial_ty_{p*}|_{t=0}$ on $\mathcal{W}$.

If we define $z:= \partial_t(y_{p_k}-y_{p^\ast})$ then 
$$\mu_k=z|_{Q_{\omega}},\quad \nu_k=z|_{\{t=0\}\times \mathcal{W}}.$$
Thus, by considering the Euler-Lagrange equation for $\mathcal{J}_{\tau,p_k}[\mu_k,\nu_k]$ at $u^\ast_{p_k}$, we obtain  

$$\int_Q\rho_a L_{p_k}(u^\ast_{p_k})L_{p_k}(v)+\int_{Q_\omega}\rho_b(u^\ast_{p_k}-z)v +\int_{\{t=0\}\times\mathcal{W}}\rho_c(u^\ast_{p_k}-z)v=0,\quad \forall v\in \mathcal{V}_{p_k}.
$$
By taking $v:= u^\ast_{p_k}-z,$ and using the linearity of $L_{p_k}$, we have $$L_{p_k}(u^\ast_{p_k})=L_{p_k}(v)+L_{p_k}(z).$$
After replacing and using the Young inequality, we obtain 
$$\int_Q\rho_a |L_{p_k}(v)|^2+\int_{Q_\omega}\rho_b|v|^2 +\int_{\{t=0\}\times\mathcal{W}}\rho_c|v|^2
\leq \frac{1}{2}\int_Q\rho_a |L_{p_k}(z)|^2+\frac{1}{2}\int_Q\rho_a |L_{p_k}(v)|^2.
$$
Thus, we obtain 
$$\int_Q\rho_a |L_{p_k}(v)|^2+\int_{Q_\omega}\rho_b|v|^2 +\int_{\{t=0\}\times\mathcal{W}}\rho_c|v|^2
\leq \int_Q\rho_a |L_{p_k}(z)|^2.
$$

Finally, by applying Theorem \ref{theo:Carleman:02} to $v$, we obtain that there exist $C>0$, $\tau_{0}\geq 1$, $h_{0}>0$, $\varepsilon >0$, depending on $\omega$, $\omega_{0}$, $T$, $m$,  $\lambda$, such that for $p_k\in \mathcal{X}_{m}$, $\tau\geq \tau_{0}(T+T^{2})$, $0<h\leq h_{0}$, $0<\delta \leq 1/2$, $\tau h(\delta T^{2})^{-1}\leq \varepsilon$, it holds
\begin{align*}
		\int_{\mathcal{W}}\tau e^{2\tau\theta(T/2)\varphi}\Big|v|_{t=T/2}\Big|^2+I_{0}(v)+J_{0}(v) 
		\leq&  C\int_Q\rho_a |L_{p_k}(z)|^2.
	\end{align*}
That is    \begin{align}\label{eq:theo:final:02}
		\int_{\mathcal{W}}\tau e^{2\tau\theta(T/2)\varphi}\Big|v|_{t=T/2}\Big|^2
		\leq&  C\int_Q\rho_a |L_{p_k}(z)|^2.
	\end{align}

On the other hand, we observe that 
\begin{align*}
    v|_{t=T/2}=&u^\ast_{p_k}|_{t=T/2}-\partial_t(y_{p_k}-y_{p^\ast})|_{t=T/2}\\
    =& u^\ast_{p_k}|_{t=T/2}-\mathcal{A}_h(y_{p_k}-y_{p^\ast})|_{t=T/2}+p_ky_{p_k}|_{t=T/2}-p^\ast y_{p^\ast}|_{t=T/2}\\
    =&(\tilde{p}_{k+1}-p^\ast)y_{p^\ast}|_{t=T/2}.
\end{align*}
Then, substituting the above equality and $L_{p_k}(z)=(p^\ast-p_k)\partial_t y_{p^\ast}$, into the left-hand and right-hand sides of \eqref{eq:theo:final:02}, respectively, we obtain
\begin{align}
		\int_{\mathcal{W}}\tau e^{2\tau\theta(T/2)\varphi}\Big|(\tilde{p}_{k+1}-p^\ast)y_{p^\ast}|_{t=T/2}\Big|^2
		\leq&  C\int_Q\rho_a |(p^\ast-p_k)\partial_t y_{p^\ast}|^2.
	\end{align}
Thus, using the lower and upper bounds for $y_{p^{\ast}}$ and $\partial_{t}y_{p^{\ast}}$ respectively, we have 
\begin{equation}\label{eq:proof:teofinal:01}
    \alpha^2 \int_{\mathcal{W}}\tau e^{2\tau\theta(T/2)\varphi}\Big|\tilde{p}_{k+1}-p^\ast\Big|^2
		\leq  C\int_Q\rho_a |p_k-p^\ast|^2.
\end{equation}

Finally, we need to compare $\rho_a$ with $e^{2\tau\theta(T/2)\varphi(x)}$. By using the Lemma \ref{Appendix-Lema1} with $q=0$ and $\theta(0)\leq (\delta T^2)^{-1}$, it follows that 
$$\int_Q\rho_a |p_k-p^\ast|^2
\leq C\left(\tau^{-\frac{1}{2}}\int_{\mathcal{W}}e^{2\tau\theta(T/2)\varphi} |p_k-p^\ast|^2+h^{-2}e^{\frac{-C'\tau}{\delta}}\int_{\mathcal{W}} |p_k-p^\ast|^2\right).$$
Now, we observe 
$$h^{-2}e^{\frac{-C'\tau}{\delta}}
\leq h^{-2}e^{-\frac{C'\tau}{T^2\delta}(1+\delta)}e^{2\theta(T/2)\tau\varphi(x)}.$$
By repeating the argument from the proof of Theorem \ref{theo:stability}, where the parameter was used $\delta$ to control $h^{-2}$  (i.e., we chose $\tau_1>0$ such that $\tau\geq \tau_1$), we have $e^{-\frac{C''\tau}{\delta T^{2}}}\leq e^{-\frac{C''\tau_1}{\delta T^{2}}}$. Taking $\delta$ sufficiently small in such a way that $\frac{\tau_1}{T^{2}\delta}=\frac{\varepsilon_{0}}{h}$, we obtain
$$\int_Q\rho_a |p_k-p^\ast|^2
\leq C(1+e^{-\frac{c''}{h}})\tau^{-\frac{1}{2}}\int_{\mathcal{W}}e^{2\tau\theta(T/2)\varphi} |p_k-p^\ast|^2,$$
and combining this with \eqref{eq:proof:teofinal:01}, yields
$$\int_Q \tau e^{2\tau\theta(T/2)\varphi} |\tilde{p}_{k+1}-p^\ast|^2
\leq C(1+e^{-\frac{c''}{h}})\tau^{-\frac{1}{2}}\int_{\mathcal{W}}e^{2\tau\theta(T/2)\varphi} |p_k-p^\ast|^2.$$
Since $T_m$ is Lipschitz and $T_m(p^\ast)=p^\ast$, we have $$|p_{k+1}-p^\ast|=|T_m(\tilde{p}_{k+1})-T_m(p^\ast)|\leq |\tilde{p}_{k+1}-p^\ast|.$$
Hence, the proof follows, which concludes the convergence of the algorithm for $\tau$ sufficiently large.
\end{proof}

\begin{remark}
    At the end of the proof of Theorem \ref{theo:convergence}, we chose $\delta=\delta(h)$ to absorb the term $h^{-2}e^{-2\theta(0)\tau\inf|\varphi|}$. By using similar arguments it is possible to prove that  $\rho_c(t,x)\leq Ce^{-\frac{c}{h}}$ holds. 
    Therefore, we observe that $\rho_c$, and consequently the term $u|_{t=0}$ represents an error that arises from the discretization procedure.
\end{remark}

\begin{remark}
    It is possible to consider the functional \begin{align}\label{FunctionalJ_alternativo}
             \tilde{\mathcal{J}}_{\tau,p}[\mu,\nu](u)=&\frac{1}{2}\int_Qe^{2\tau\theta\varphi}|L_p(u)|^2+\frac{1}{2}\int_{Q_\omega}(\tau\theta)^3e^{2\tau\theta\varphi}|u-\mu|^2+\\
             &
             e^{-\tau\theta(0)\inf|\varphi|}\left(\frac{1}{2}\int_Qe^{2\tau\theta(T/2)\varphi}|L_p(u)|^2+\frac{1}
             {2}\int_{\mathcal{W}}\Big|u|_{t=0}-\nu\Big|^2\right),\notag  
\end{align}
as an alternative to $\mathcal{J}_{\tau,p}$ in the Algorithm \ref{Alporithmo_Reconstruccion} since $\tilde{\mathcal{J}}_{\tau,p}$ is an upper bound of $\mathcal{J}_{\tau,p}$ when $\delta(h)$; and thus obtaining the same result as in Theorem \ref{theo:convergence}. The advantage of using \eqref{FunctionalJ_alternativo} instead is that it does not contain singular terms $o(h^{-2})$. Moreover, we can explicitly see how $\delta(h)\to 0$ to control the error term 
\[ e^{-\tau\theta(0)\inf|\varphi|}\left(\frac{1}{2}\int_Qe^{2\tau\theta(T/2)\varphi}|L_p(u)|^2+\frac{1}
{2}\int_{\mathcal{W}}\Big|u|_{t=0}-\nu\Big|^2\right),\] since $\theta(0)= \frac{1}{T^2(1+\delta)\delta}$. 
Namely, $\delta(h)\in (0,\frac12)$ is such that $\delta(h)\to 0$ and $\frac{h}{\delta}\to 0,$ when $h$ goes to zero. That is, we need to take $\delta$ to zero more slowly than $h$. For instance, by taking $\delta(h)=h^{\sigma},$ with $\sigma\in (0,1)$, when $h$ is sufficiently small.         
\end{remark}


\section{Concluding remarks and perspectives}\label{sec:concluding}

    In this work, we adapted the methodology from \cite{IY-1998} to the semi-discrete setting. This involved the development of a new Carleman estimate for the semi-discrete parabolic operator, as previous Carleman estimates for these operators did not include the second-order term on the left-hand side. This omission was due to their primary applications in controllability problems. Moreover, when the diffusive coefficient is time-independent, we established Lipschitz stability with respect to the measurements. It is known that an algorithm based on Carleman estimates with two parameters are difficult to implement. For this reason, several efforts are focused on the development of the Carleman estimate with a weight function with only one parameter \cite{KJZ:2020}. For instance, reconstruction algorithms are studied for the wave equation in \cite{BDBE:2013,BE:2013,BEO-2015}. 
    As an open problem, one could aim to adapt the presented algorithm to use a Carleman weight with only a single parameter. This might involve incorporating additional measurements to facilitate a numerical implementation of the reconstruction algorithm in that setting, following the developments from \cite{BEO-2015}.
    
    Regarding the results presented in \cite{IY-1998}, we observe that they also establish a stability result based on boundary measurements. To achieve a similar result in the semi-discrete setting, it is essential to develop a semi-discrete Carleman estimate with boundary observation. In this direction, to the best of our knowledge, only a few works address Carleman estimates with boundary data; see, for instance, \cite{LOPD:2023,ZY-2024} for the discrete Laplacian operator and \cite{CLNP2022} for a semi-discrete fourth-order parabolic operator. Therefore, as a first step toward incorporating boundary observation, one must derive a semi-discrete Carleman estimate for a semi-discrete parabolic operator with boundary data. Furthermore, motivated by \cite{CLNP2022,WZ:2024}, it would be interesting to explore inverse problems for higher-order operators using semi-discrete Carleman estimates.
    
    In \cite{BDLR:2007}, the results of controllability and inverse problems were obtained for parabolic operators with a discontinuous diffusion coefficient. A natural extension of our work would be to establish the stability of a coefficient inverse problem when the diffusive function is discontinuous. A promising approach could be to adapt the methodology from \cite{N:2014}, where a Carleman estimate was developed for a semi-discrete parabolic operator with discontinuous diffusive coefficient in the one-dimensional setting and applied to obtain controllability results. Hence, the first step is to extend this methodology to arbitrary dimensions and subsequently to adapt it to the study of inverse problems.
    
    Recently, the Lipschitz stability for the discrete inverse random source problem and the Hölder stability for the discrete Cauchy problem have been obtained in \cite{Wu_2024} in the one-dimensional setting. In turn, a Carleman estimate for the semi-discrete stochastic parabolic operator is obtained in arbitrary dimensions, implying a controllability result \cite{LPP:2024}. We note that the methodology developed here cannot be used in the stochastic case, although the discrete setting can be used to extend into arbitrary dimension the semi-discrete inverse problem studied in \cite{Wu_2024}. We refer to \cite{QL:YW:2024} and references therein for stochastic inverse problems in the continuous framework.
    
    The inverse problem of coefficient identification with time discretization is addressed in \cite{Klibanov2024_tiempo_discreto}. A natural extension of this work would be to consider the fully discrete problem in both space and time. Achieving this would require the development of a fully discrete version of the Carleman estimates, potentially by adapting the techniques presented in \cite{GCHS2021,LMZP2023}. Moreover, exploring the extension to systems of parabolic equations, as investigated in \cite{MFG_Klibavov_2023} with a boundary measurement, presents another compelling research direction. Finally, the study of numerical reconstruction schemes similar to those presented in \cite{CNumerico_Klibanov_2025} would also be a valuable contribution.
    
    There are works that address the reconstruction of $p$ from equation \eqref{system:coefficient:02}. For example, in \cite{YZ:2001}, the problem of simultaneously recovering the potential $p$ and the initial condition of the system is presented. Stability is established for both the parameter $p$ and the initial condition, which is similar to the stability we obtain in Theorem \ref{theorem:convergence}, but without the measurement at $t = 0$, and it is achieved regardless of whether the solutions have different initial conditions. This stability for $p$ is used in \cite{YZ:2001} to obtain logarithmic stability for the initial data, a matter that we do not address and which would remain as an open problem in the present semi-discrete framework. We believe that at least two difficulties of the technique must be addressed: the first is the existence of a semi-discrete version of ``the method of logarithmic convexity'' (see \cite[equation 2.20]{YZ:2001}), and the second is the difficulty that the stability for $p$ involves the initial data in the semi-discrete framework.
    
    Regarding the algorithm presented in this work, it does not involve a Tikhonov-type regularization, unlike the approaches in \cite{YZ:2001} and \cite{KJZ:2020}. However, it incorporates a Carleman weight within the functional to be minimized, which is similar to the convexification method presented in \cite{KJZ:2020}. Unfortunately, since this Carleman weight depends on two parameters, achieving a robust implementation of the optimization problem is particularly challenging. This difficulty arises from the extreme sensitivity of the double exponential inherent in the Carleman weight, a point also discussed in \cite{KJZ:2020} and \cite{BE:2013}.

    In conclusion, we have successfully adapted the stability results of the continuous case to the semi-discrete one, but to address the singular error terms $\mathcal{O}(h^{-2})$, which emerge from the semi-discrete Carleman inequality, we have had to establish restrictions on the parameter $\delta$. The parameter $\delta$ represents the regularization of the Carleman weight function at the singularity points $t=0$ and $t=T$. Consequently, to achieve stability results that resemble the continuous case, we require that the semi-discrete weight function converges to the continuous one as $\delta$ tends to zero, while simultaneously ensuring that $\delta$ continues to satisfy the constraints of the semi-discrete Carleman inequality. Additionally, we present a reconstruction algorithm for the inverse coefficient problem. This algorithm can be adapted to the continuous case and, to our knowledge, has not been previously presented in the literature.

\section*{Acknowledgments}
R. Lecaros, J. L\'opez-R\'ios, and A. A. P\'erez have been partially supported by the Math-Amsud project	CIPIF 22-MATH-01. R. Lecaros was partially supported by FONDECYT (Chile) Grant 1221892. J. L\'opez-R\'ios acknowledges support by Vicerrectoría de Investigación y Extensión of Universidad Industrial de Santander.  A. A. P\'erez was supported by Vicerrector\'ia de Investigaci\'on y postgrado, Universidad del B\'io-B\'io, proyect IN2450902 and FONDECYT Grant 11250805.

\bibliographystyle{abbrv}
\bibliography{references}
\end{document}